\documentclass[11pt,reqno]
{amsart}
\usepackage{amsmath,epsfig,graphicx,color}
\usepackage{subfig}
\usepackage{dsfont}
\usepackage{ifthen}
\usepackage{amssymb,latexsym}
\usepackage{hyperref}
\hypersetup{
urlcolor=black, 
  menucolor=black, 
  citecolor=black, 
  anchorcolor=black, 
  filecolor=black, 
  linkcolor=black, 
  colorlinks=true,
}
\usepackage{comment}
\usepackage[numbers,sort&compress]{natbib} 
\newcommand{\cip}{\stackrel{\P}{\rightarrow}}

\newcommand{\eid}{\stackrel{\rm d}{=}}

\newcommand{\x}{\mathbf{x}}
\newcommand{\tx}{\tilde{\mathbf{x}}}
\newcommand{\y}{\mathbf{y}}

\usepackage{mathtools}
\mathtoolsset{showonlyrefs}
\usepackage{multirow}
\usepackage[table]{xcolor}

\definecolor{darkblue}{rgb}{.2, 0.2,.8}
\definecolor{darkgreen}{rgb}{0,0.5,0.3}
\definecolor{darkred}{rgb}{.8, .1,.1}

\newcommand{\red}{\color{darkred}}
\newcommand{\blue}{\color{darkblue}}

\newtheorem{lemma}{Lemma}[section]

\newtheorem{theorem}[lemma]{Theorem}

\newtheorem{proposition}[lemma]{Proposition}
\newtheorem{definition}[lemma]{Definition}
\newtheorem{corollary}[lemma]{Corollary}
\newtheorem{example}[lemma]{Example}
\newtheorem{exercise}[lemma]{Exercise}
\newtheorem{remark}[lemma]{Remark}
\newtheorem{fig}[lemma]{Figure}
\newtheorem{tab}[lemma]{Table}

\newcommand{\cid}{\stackrel{\rm d}{\rightarrow}}

\newcommand{\bth}{\begin{theorem}}
\newcommand{\ethe}{\end{theorem}}

\newcommand{\bre}{\begin{remark}\em }
\newcommand{\ere}{\end{remark}}

\newcommand{\ble}{\begin{lemma}}
\newcommand{\ele}{\end{lemma}}

\newcommand{\bde}{\begin{definition}}
\newcommand{\ede}{\end{definition}}
\newcommand{\bco}{\begin{corollary}}
\newcommand{\eco}{\end{corollary}}

\newcommand{\bpr}{\begin{proposition}}
\newcommand{\epr}{\end{proposition}}

\newcommand{\bexer}{\begin{exercise}}
\newcommand{\eexer}{\end{exercise}}

\newcommand{\bexam}{\begin{example}}
\newcommand{\eexam}{\end{example}}

\newcommand{\bfi}{\begin{fig}}
\newcommand{\efi}{\end{fig}}

\newcommand{\btab}{\begin{tab}}
\newcommand{\etab}{\end{tab}}

\newcommand{\sign}{{\rm sign}}

\newcommand{\Var}{\operatorname{Var}}

\newcommand{\Cov}{\operatorname{Cov}}

\newcommand{\rhs}{right-hand side}

\newcommand{\beao}{\begin{eqnarray*}}
\newcommand{\eeao}{\end{eqnarray*}\noindent}

\newcommand{\beam}{\begin{eqnarray}}
\newcommand{\eeam}{\end{eqnarray}\noindent}

\newcommand{\beqq}{\begin{equation}}
\newcommand{\eeqq}{\end{equation}\noindent}

\newcommand{\bce}{\begin{center}}
\newcommand{\ece}{\end{center}}

\newcommand{\barr}{\begin{array}}
\newcommand{\earr}{\end{array}}

\newcommand{\vague}{\stackrel{\lower0.2ex\hbox{$\scriptscriptstyle
                    \it{v} $}}{\rightarrow}}
\newcommand{\weak}{\stackrel{\lower0.2ex\hbox{$\scriptscriptstyle
                    \it{w} $}}{\rightarrow}}
\newcommand{\what}{\stackrel{\lower0.2ex\hbox{$\scriptscriptstyle
                    \it{\hat{w}} $}}{\rightarrow}}

\newcommand{\bdis}{\begin{displaymath}}
\newcommand{\edis}{\end{displaymath}\noindent}

\newcommand{\N}{\mathbb{N}}
\newcommand{\R}{\mathbb{R}}

\newcommand{\nto}{n\to\infty}

\newcommand{\wt}{\widetilde}

\newcommand{\vep}{\varepsilon}

\newcommand{\st}{such that}

\newcommand{\bfS}{{\bf S}}

\newcommand{\bfI}{{\bf I}}

\newcommand{\E }{{\mathbb E}}
\renewcommand{\P }{{\mathbb P}}

\newcommand{\1}{\mathds{1}}

\allowdisplaybreaks

\DeclareMathOperator{\e}{e}

\newcommand{\tr}{\operatorname{tr}}

\newcommand{\bfSigma}{{\mathbf \Sigma}}

\newcommand{\jh}[1]{{\red #1}}

\renewcommand{\S}{{\mathbf S}}

\makeatletter
\let\save@mathaccent\mathaccent
\newcommand*\if@single[3]{%
  \setbox0\hbox{${\mathaccent"0362{#1}}^H$}%
  \setbox2\hbox{${\mathaccent"0362{\kern0pt#1}}^H$}%
  \ifdim\ht0=\ht2 #3\else #2\fi
  }
\newcommand*\rel@kern[1]{\kern#1\dimexpr\macc@kerna}
\newcommand*\widebar[1]{\@ifnextchar^{{\wide@bar{#1}{0}}}{\wide@bar{#1}{1}}}
\newcommand*\wide@bar[2]{\if@single{#1}{\wide@bar@{#1}{#2}{1}}{\wide@bar@{#1}{#2}{2}}}
\newcommand*\wide@bar@[3]{%
  \begingroup
  \def\mathaccent##1##2{%
    \let\mathaccent\save@mathaccent
    \if#32 \let\macc@nucleus\first@char \fi
    \setbox\z@\hbox{$\macc@style{\macc@nucleus}_{}$}%
    \setbox\tw@\hbox{$\macc@style{\macc@nucleus}{}_{}$}%
    \dimen@\wd\tw@
    \advance\dimen@-\wd\z@
    \divide\dimen@ 3
    \@tempdima\wd\tw@
    \advance\@tempdima-\scriptspace
    \divide\@tempdima 10
    \advance\dimen@-\@tempdima
    \ifdim\dimen@>\z@ \dimen@0pt\fi
    \rel@kern{0.6}\kern-\dimen@
    \if#31
      \overline{\rel@kern{-0.6}\kern\dimen@\macc@nucleus\rel@kern{0.4}\kern\dimen@}%
      \advance\dimen@0.4\dimexpr\macc@kerna
      \let\final@kern#2%
      \ifdim\dimen@<\z@ \let\final@kern1\fi
      \if\final@kern1 \kern-\dimen@\fi
    \else
      \overline{\rel@kern{-0.6}\kern\dimen@#1}%
    \fi
  }%
  \macc@depth\@ne
  \let\math@bgroup\@empty \let\math@egroup\macc@set@skewchar
  \mathsurround\z@ \frozen@everymath{\mathgroup\macc@group\relax}%
  \macc@set@skewchar\relax
  \let\mathaccentV\macc@nested@a
  \if#31
    \macc@nested@a\relax111{#1}%
  \else
    \def\gobble@till@marker##1\endmarker{}%
    \futurelet\first@char\gobble@till@marker#1\endmarker
    \ifcat\noexpand\first@char A\else
      \def\first@char{}%
    \fi
    \macc@nested@a\relax111{\first@char}%
  \fi
  \endgroup
}
\makeatother

\renewcommand{\bar}{\widebar}

\begin{document}
\bibliographystyle{acm}
\title[Joint convergence of point process and Frobenius norm]{Asymptotic independence of point process and Frobenius norm of a large sample covariance matrix}
\thanks{Johannes Heiny's and Carolin Kleemann's research was partially supported by the Deutsche Forschungsgemeinschaft (DFG) via RTG 2131 High-dimensional Phenomena in Probability -- Fluctuations and Discontinuity.}

\author[J. Heiny]{Johannes Heiny}
\address{Fakult\"at f\"ur Mathematik,
Ruhruniversit\"at Bochum,
Universit\"atsstrasse 150,
D-44801 Bochum,
Germany}
\email{johannes.heiny@rub.de}
\author[C. Kleemann]{Carolin Kleemann}
\address{Fakult\"at f\"ur Mathematik,
Ruhruniversit\"at Bochum,
Universit\"atsstrasse 150,
D-44801 Bochum,
Germany}
\email{carolin.kleemann@rub.de}

\begin{abstract}
A joint limit theorem for the point process of the off-diagonal entries of a sample covariance matrix $\mathbf{S}$, constructed from $n$ observations of a $p$-dimensional random vector with iid components, and the Frobenius norm of $\mathbf{S}$ is proved. In particular, assuming that $p$ and $n$ tend to infinity we obtain a central limit theorem for the Frobenius norm in the case of finite fourth moment of the components and an infinite variance stable law in the case of infinite fourth moment. Extending a theorem of Kallenberg, we establish asymptotic independence of the point process and the Frobenius norm of $\mathbf{S}$. To the best of our knowledge, this is the first result about joint convergence of a point process of dependent points and their sum in the non-Gaussian case. 
\end{abstract}
\maketitle

\section{Introduction}

Over recent years the analysis of high-dimensional data has emerged as an important and active research area driven by a wide range of applications in various fields such as genomics, medical imaging, signal processing, financial engineering and social science. 
To study large data sets, for instance, in brain connectivity analysis or in gene expression analysis (see \cite{shaw2006intellectual, zhang2008class, fukushima2013diffcorr}), knowledge of the dependence structure plays a central role. Interpreting the data as observations of a $p$-dimensional random vector, dependence between the components of the vector is  often estimated by covariance/correlation statistics  and different functions are used to aggregate these estimates of the pairwise dependencies. For example,  \cite{schott2005, Bandedness, yaoetal2017} and \cite{Leung2018} propose sum-type tests based on the Frobenius norm which are usually powerful against dense alternatives.  
Further very popular methods of aggregating  estimates of the pairwise dependencies are  maximum-type tests, which have good power properties against sparse alternatives and have been investigated for various
covariance/correlation statistics in \cite{jiang:2004b, zhou:2007, liu:lin:shao:2008, yaoetal2017, cai:jiang:2012, han:chen:liu:2017, drton2020high, heiny:mikosch:yslas:2021} 
and \cite{he:xu:wu:pan:2021} among others.
Since in practice it is difficult to decide whether the underlying covariance matrix is sparse or dense, it is useful to combine these two types of test statistics  to cover both cases \cite{ feng2022max, fan:liao:yao:2015, he:xu:wu:pan:2021, qing:pan:2017, chen2022asymptotic}. Therefore, an understanding of the joint asymptotic behavior of these test statistics is needed.

The objective of this paper is to contribute to this line of research by providing asymptotic theory for the joint distribution of sum-type statistics and generalized maximum-type statistics of a large sample covariance matrix.
\medskip

For a sample $\x_1,\ldots, \x_n$ from the $p$-dimensional population $\x$ with independent and identically distributed (iid) components with mean zero and variance one, 
the sample covariance matrix $\S$ is given by
\begin{align*}
    \S:=\S_n:=\frac{1}{n}\sum_{t=1}^n\x_t \x_t^\top =\big(S_{ij})_{1\le i,j\le p}\,.
\end{align*}
Throughout this paper, we assume that the dimension $p=p_n$ is a positive integer sequence tending to infinity together with the sample size $n$. Thus, the $p\times p$-matrix $\S$ is a high-dimensional
random matrix whose asymptotic properties are used, for example, in independence testing. 
Sum-type and maximum-type statistics based on $\S$ are given by the (squared) Frobenius norm and the maximum off-diagonal entry of $\S$, 
\begin{equation}\label{eq:testst}
    \|\S\|_F^2:=\sum_{i,j=1}^p S_{ij}^2 =\tr(\S^2) \quad \text{ and } \quad \max_{1\le i<j\le p} S_{ij}\,.
\end{equation}
We are interested in the joint limiting distribution of $\tr(\S^2)$ and the sequence of point processes of the off-diagonal entries of $\S$
\begin{align}\label{defpp}
  N_n:=\sum_{1\leq i<j\leq p}\varepsilon_{d_p(\sqrt{n}S_{ij}-d_p)} , 
\end{align}
 where $\vep_x$ denotes the Dirac measure in $x\in \R$ and 
\begin{align}\label{eq:defdp}
d_p:=\sqrt{2\log \tilde{p}}-\frac{\log\log\tilde{p}+\log 4\pi}{2(\log\tilde{p})^{1/2}} \qquad \text{for } \tilde{p}:=p(p-1)/2\,.
\end{align}
It is well-known that the point process in \eqref{defpp} contains information about all order statistics of the $S_{ij}$'s (see \cite{embrechts:kluppelberg:mikosch:1997}). The distribution of the maximum can be recovered from the identity $\{N_n((x,\infty))=0\}=\{\max_{i<j} d_p(\sqrt{n}S_{ij}-d_p)\le x\}$, $x\in \R$. In this sense, the point process $N_n$ is a natural and meaningful generalization of the maximum. 

  In Table~\ref{tab:overview} an overview of the available results about the convergence of $\max S_{ij}, N_n$ and $\tr(\S^2)$ and the novel contributions of this paper (marked in blue) is given. For the reader's convenience, the table  contains the limit distributions themselves and precise references. It is worth mentioning that the distinction between finite and infinite fourth moment of $X$, which is a generic random variable with the same distribution as the components of $\x$, results from the sum-type statistic $\tr(\S^2)$.

\begin{table}[h]
    \centering
   \def\arraystretch{2.0}
\begin{tabular}[h]{|c|c|c|c|c|}
\hline
\rowcolor{lightgray}
&$\Var(X)<\infty$, $\E[X^{4}]=\infty$& $\E[X^4]<\infty$  \\
\hline
$\max S_{ij}$ & Gumbel distribution, \cite[Thm. 3.2]{heiny:mikosch:yslas:2021} & Gumbel distribution, \cite[Lem. 3.2]{jiang:2004b} \\
\hline
$N_n$ & Poisson process, \cite[Thm. 3.2]{heiny:mikosch:yslas:2021}  & Poisson process $N$, \cite[Thm. 3.2]{heiny:mikosch:yslas:2021} \\
\hline
$\tr(\S^2)$& {\blue Stable distribution, Thm. \ref{thm:non-clt}}  & {\blue Normal distribution $Z$, Thm. \ref{thm:clt}} \\
\hline
$(N_n,\tr(\S^2))$& open & {\blue $(N,Z)$ independent, Thm. \ref{thm:ind}} \\
 \hline
\end{tabular}
\caption{Overview of the asymptotic results about $\max S_{ij}, N_n$ and $\tr(\S^2)$.}
    \label{tab:overview}
\end{table}
\subsection{Related literature on sums, maxima and point processes} 
The joint behavior of the sum and the maximum of a sequence of real-valued random variables has been studied before, motivated for example by the evaluation of wind speed data, which is usually available in the form of the maximum wind speed and the average wind speed during a day or an hour. For iid random variables $(Y_i)_{i\ge 1}$ we set
\begin{align*}
    S_n:=\sum_{i=1}^n Y_i\quad \text{and}\quad M_n:=\max_{1\leq i\leq n} Y_i.
\end{align*}
If the distribution function $F$ of $Y_1$ belongs to the sum domain of attraction of the normal distribution and the maximum domain of attraction of an extreme value distribution, \cite{chow1978sum} showed that $(S_n,M_n)$ converges in distribution to a limit $(S,M)$, where $S$ and $M$ are independent and not degenerated. They also proved that if 
\begin{align*}
    1-F(x)= q_+\,x^{-\alpha}L(x) \quad \text{and} \quad F(-x)= q_- \,x^{-\alpha}L(x),\quad\quad x>0\,, \alpha\in(0,2)\,,
\end{align*}
where $L$ is a slowly varying function and $0<q_+\le q_++q_-= 1$, then $(S_n,M_n)$ converges to a limit $(S,M)$, where $S$ and $M$ are dependent and they provide a hybrid characteristic distribution function of $(S,M)$.

 The papers \cite{anderson1991joint, anderson1995sums, hsing1995note} generalized these results to strongly mixing stationary random variables. For stationary normal random variables \cite{ho1996asymptotic} and \cite{mccormick2000asymptotic} proved asymptotic independence under certain correlation assumptions. Recently, asymptotic independence of a quadratic form in and the maximum of independent random variables was proved in \cite{chen2022asymptotic} and asymptotic independence of the sum and the maximum of dependent normal random variables that need not be stationary or strongly mixing but fulfill conditions on the smallest and largest eigenvalue of their covariance matrix was shown in \cite{feng2022asymptotic}.
 
For a triangular array of normal distributed random variables the asymptotic independence of the point process of exceedances and the partial sum was considered in \cite{hu2009joint, tan2011joint} and extended to the multivariate case in \cite{peng2012joint}. Moreover, \cite{guo:lu:2023} established asymptotic independence of the point processes of clusters and the partial sums of bivariate stationary Gaussian triangular arrays. Asymptotic independence of other quantities derived from the sample covariance matrix $\S$ has also been considered in a variety of settings. For example, the asymptotic independence of the maximum of sample correlations and the sum of the squared sample correlations between the residuals from the ordinary least squares is proven in \cite{feng2022max}, while \cite{li2020asymptotic} showed asymptotic independence of the largest sample eigenvalues and the trace of $\S$. 

\subsection{Structure of this paper}
This paper is structured as follows. Section~\ref{sec:main} contains our main results about the point process of the off-diagonal entries of $\S$ and the Frobenius norm of $\S$. Under finite fourth moment of $X$  the Frobenius norm satisfies a CLT (Theorem~\ref{thm:clt}), while in the case of infinite fourth moment we obtain a stable limit law (Theorem~\ref{thm:non-clt}). 
The main result of this paper is Theorem~\ref{thm:ind}, which shows asymptotic independence of the point process and the Frobenius norm of $\bfS$. The challenges in the proof of this result and our novel technical contributions are outlined in Section~\ref{sec:strategy}, while Section~\ref{sec:test} presents an application to independence testing. The proofs are deferred to Section~\ref{sec:4} and helpful auxiliary results are given in Section~\ref{sec:5}.

\subsection{Notation}
Convergence in distribution (resp.\ probability) is denoted by $\cid$ (resp.\ $\cip$), equality in distribution by $\eid$, and unless explicitly stated otherwise all limits are for $\nto$. For sequences $(a_n)_n$ and $(b_n)_n$ we write $a_n=O(b_n)$ if $a_n/b_n\leq C$ for some constant $C>0$ and every $n\in\N$, and $a_n=o(b_n)$ if $\lim_{n\to\infty} a_n/b_n=0$. Additionally, we use the notation
$a_n\sim b_n$ if $\lim_{n\to\infty} a_n/b_n=1$.

\section{Main results}\label{sec:main}

Consider a sample $\x_1,\ldots, \x_n$ from the $p$-dimensional population $\x$ with iid components with generic element $X$ satisfying $\E[X]=0$ and $\E[X^2]=1$. We will work in the high-dimensional setting, where the dimension $p=p_n$ is some positive integer sequence tending to infinity as $\nto$. 
We aim to study the joint asymptotic behavior of the point processes $N_n$ of the off-diagonal entries (see \eqref{defpp}) and the Frobenius norm of the sample covariance matrix $\S=(S_{ij})=\frac{1}{n}\sum_{t=1}^n\x_t \x_t^\top$.

\subsection{Asymptotic distributions of point process and Frobenius norm of the sample covariance matrix} \label{sec:2}

First we consider the convergence in distribution of the sequence of point processes $N_n$ towards a Poisson random measure. For a detailed background on weak convergence of point processes we refer to \cite{resnick:2007, daley:verejones:1988}.
 The following result can be found in \cite[Theorem 3.2]{heiny:mikosch:yslas:2021} and \cite[Theorem 4.1]{heiny2023maximum}. 

\begin{theorem}\cite[Theorem~3.2]{heiny:mikosch:yslas:2021}\label{lem:pp}
Assume that there exist $s>2$ and $\varepsilon>0$ \st\ $\E[|X|^{s+\varepsilon}]< \infty$ and let $p=p_n\to\infty$ satisfy $p=O(n^{(s-2)/4})$, as $\nto$.
Then it holds that $N_n\cid N$,
where $N$ is a Poisson random measure with mean measure $\mu(x,\infty)=\e^{-x}$ for $x\in\R$.
\end{theorem}
   The Poisson random measure $N$ with mean measure $\mu(x,\infty)=\e^{-x}$ for $x\in\R$ has the representation
   \begin{align}\label{eq:representationN}
       N=\sum_{i=1}^{\infty} \vep_{-\log \Gamma_i},
   \end{align}
   where $\Gamma_i= E_1+\cdots+E_i$, $i\ge 1$, and $(E_i)$ is a sequence of iid standard exponential random variables \cite{resnick:2007}.
   
Our second object of interest is the squared Frobenius norm of the sample covariance matrix, that is $\|\S\|_F^2:=\sum_{i,j=1}^p S_{ij}^2 =\tr(\S^2)$.
In order to study its asymptotic distribution we define 
\begin{align} \label{clt}
 Z_n:=\frac{\tr(\S^2)-\mu_n}{\sigma_n},   
\end{align}
where 
\begin{align}
\sigma_n^2 &:=\Big(\frac{p}{n}+2\frac{p^2}{n^2}+\frac{p^3}{n^3}\Big)4(\E[X^4]-1)+4\frac{p^2}{n^2}\,, \label{eq:defsigma}\\
\mu_n &:=\E[\tr(\S^2)]= \frac{p}{n}(n+\E[X^4]-2)+\frac{p^2}{n}\,. \label{eq:defmu}
\end{align}
The next result provides a CLT for $\tr(\S^2)$ in the general case that $p_n\to\infty$ and $\E[X^4]<\infty$.
\begin{theorem}\label{thm:clt}
If $\E[X^4]<\infty$ and $p=p_n\to\infty$, then as $n\to\infty$ we have $Z_n\cid Z$ for a standard normal random variable $Z$. 
\end{theorem}
\begin{remark}{\em 
 Under special assumptions on $X$ and the growth of $p$, the behavior of $Z_n$ can be deduced from CLTs for so-called linear spectral statistics of sample covariance matrices: for example, \cite[Theorem 9.10]{bai:silverstein:2010} and \cite[Theorem 1.4]{pan2008central} in the case $p/n\to C\in (0,\infty)$ and $\E[X^4]<\infty$, or \cite[Theorem 3.1]{qiu2021asymptotic} in the case $n^2/p=O(1)$ if $\E[|X|^{6+\delta}]<\infty$ for some $\delta>0$. 
 }\end{remark}
In contrast to the convergence of the point processes $N_n$ in Theorem~\ref{lem:pp}, the CLT in Theorem~\ref{thm:ind} requires the existence of the fourth moment of $X$. To characterize the asymptotic distribution of $\tr(\S^2)$ in the case $\E[X^4]=\infty$ we need to assume that $X$ is a regularly varying random variable. 
We say that a random variable $X$ (or its distribution) is {\it regularly varying with index $\alpha>0$} if
\begin{align*}
    \P(|X|>x)=x^{-\alpha} \,L(x)\,, \qquad x>0\,,
\end{align*}
where $L$ is a slowly varying function, i.e., $\lim_{x\to\infty} L(tx)/L(x)=1$ for $t>0$. Examples of regularly varying distributions are the Pareto distribution with parameter $\alpha$ and the $t$-distribution with $\alpha$ degrees of freedom.

For a regularly varying random variable $X$ with index $\alpha$ it holds that $\E[|X|^\beta]<\infty$ if $\beta<\alpha$ and $\E[|X|^\beta]=\infty$ if $\beta>\alpha$.  It is well--known that the sequence  $(a_{k})_k$ defined through 
$$a_k:=\inf \{x\in \R: \P(|X|>x)\le 1/k \}\,, \qquad k\ge 1\,,$$ 
is of the form $a_k=k^{1/\alpha}\ell(k)$, where $\ell$ is a slowly varying function. For further properties of regularly varying functions we refer to \cite{bingham:goldie:teugels:1987, samorodnitsky:taqqu:1994}. In the next theorem we consider the case that $X$ has finite variance but infinite fourth moment.

\begin{theorem}\label{thm:non-clt}
Let $X$ have a regularly varying distribution with index $\alpha\in(2,4)$ and assume that $p=p_n\to\infty$. Then it holds that
\begin{align}\label{eq:nonclt}
\frac{n^2}{a_{np}^4}\tr(\S^2)-\frac{2n(n+p-2)}{a_{np}^4}\tr(\S)+\frac{np(n+p-2)}{a_{np}^4}\cid \zeta_{\alpha/4},\qquad \nto\,,
\end{align}
where $\zeta_{\alpha/4}$ is a non-degenerated, $\alpha/4$-stable random variable with characteristic function 
\begin{align*}
    \E[e^{{\rm i}t\zeta_{\alpha/4}}]=\exp\Big({\rm i}tc_\alpha+\frac{\alpha}{4}\int_0^\infty\big(e^{itx}-1-\frac{{\rm i}tx}{1+x^2}\big)x^{-(\alpha/4+1)}dx\Big),
\end{align*}
where ${\rm i}$ is the imaginary unit and $c_\alpha$ is a constant only depending on $\alpha$.
\end{theorem} 
In the proof of Theorem~\ref{thm:non-clt} we show that
\begin{align*}
 \frac{n^2}{a_{np}^4}\tr(\S^2)-\frac{2n(n+p-2)}{a_{np}^4}\tr(\S)+\frac{np(n+p-2)}{a_{np}^4}=   \frac{1}{a_{np}^4}\sum_{i=1}^p\sum_{t=1}^n X_{it}^4+ o_{\P}(1)\,.
\end{align*}
Noticing that $\sum_{i,t} X_{it}^4$ is a sum of iid  regularly varying random variables with index $\alpha/4\in(1/2,1)$, we obtain an $\alpha/4$-stable limit distribution after proper normalization \cite{samorodnitsky:taqqu:1994}. 

\begin{remark}\label{rem:np}{\em 
 (a) In some cases it is possible to replace $\tr(\S)$ by its expectation $\E[\tr(\S)]=p$ in \eqref{eq:nonclt}.  For example,  if $\lim_{\nto} p/n \in(0,\infty)$, we get for $\delta>0$
    \begin{align*}
       \P\Big(\Big|\frac{2n(n+p-2)}{a_{np}^4}\big(\tr(\S)-\E[\tr(\S)]\big)\Big|>\delta\Big)&= \P\Big(\Big|\frac{2(n+p-2)}{a_{np}^4}\sum_{i=1}^p\sum_{t=1}^n (X_{it}^2-1)\Big|>\delta\Big)\\
       &\sim np\,\P\Big(|X^2-1|>\tfrac{\delta a_{np}^4}{2(n+p-2)}\Big) \to 0\,, \qquad \nto\,,
    \end{align*}
   where Theorem A1 of \cite{davis2016asymptotic} and the fact that $X^2-1$ is regularly varying with index $\alpha/2$ were used for the asymptotic equivalence in the last line.\\   
(b) For $\alpha=4$ the limit in \eqref{eq:nonclt} does not hold in general, which we shall illustrate in the case $\E[X^4]<\infty$. As $\tr(\S)$ is a sum of iid random variables with finite variance, it satisfies a CLT. Furthermore, Theorem~\ref{thm:clt} states that $\tr(\S^2)$ is also asymptotically normal. Along the lines of the proof of Theorem~\ref{thm:clt} one can show joint asymptotic normality of $\tr(\S)$ and $\tr(\S^2)$. For the sake of brevity we omit a proof, but we mention that in the special case $\lim_{\nto} p/n \in(0,\infty)$ joint asymptotic normality of $\tr(\S)$ and $\tr(\S^2)$ was established in \cite[Lemma~2.2]{wang:yao:2013}.
}\end{remark}

\subsection{Joint limiting distribution of point process and Frobenius norm of the sample covariance matrix}\label{sec:3}
In this subsection we are interested in the joint limiting distribution of $Z_n$ in \eqref{clt} and the point process $N_n$ in \eqref{defpp}.
For this purpose we start by giving a definition of asymptotic independence; c.f.\  \cite[p.~284]{hu2009joint}.
For the Poisson random measure $N$ defined in \eqref{eq:representationN} we will need the collection of sets
\begin{align*}
\mathcal{B}_N:=\{B\; \text{bounded Borel set}: N(\partial B)=0\},
\end{align*}
where $\partial B$ is the boundary of $B$. 


\begin{definition}\label{def:2.5}
Let $(Y_n)_n$ be a sequence of real-valued random variables, which converges to the random variable $Y$ in distribution. Additionally, let $(N_n)_n$ be a sequence of point processes on $\R$, which converges to the point process $N$ in distribution. We call $(Y_n)_n$ and $(N_n)_n$ asymptotically independent, if and only if for every $y\in\R$, $B_1,\ldots, B_k\in\mathcal{B}_N$ and $l_1,\ldots, l_k\in\N_0:=\N\cup \{0\}$
\begin{align*}
\lim\limits_{n\to\infty}\P(Y_{n}\leq y,\, N_{n}(B_1)\leq l_1,\ldots,\,N_{n}(B_k)\leq l_k)=\P({Y}\leq y)\P({N}(B_1)\leq l_1,\ldots,\, {N}(B_k)\leq l_k).
\end{align*}
\end{definition}
Now we state our main result about the joint convergence of $(Z_n,N_n)$.
\begin{theorem}\label{thm:ind}
Assume that there exist $s\ge 4$ and $\varepsilon>0$ \st\ $\E[|X|^{s+\varepsilon}]< \infty$ and let $p=p_n\to\infty$ satisfy $p=O(n^{(s-2)/4})$, as $\nto$.
 Then it holds for the standardized traces $(Z_n)_n$ defined in \eqref{clt} and the point processes $(N_n)_n$ in \eqref{defpp} that
\begin{align*}
(Z_n,N_n)\cid (Z,N)\,, \qquad \nto\,,
\end{align*}
where $Z\sim \mathcal{N}(0,1)$, $N$ is a Poisson random measure with mean measure $\mu(x,\infty)=\e^{-x}, x\in\R$, and $Z$ and $N$ are independent.
\end{theorem}
Theorem~\ref{thm:ind} only requires the conditions of Theorems~\ref{lem:pp} and \ref{thm:clt} which are both formulated under minimal assumptions, c.f.\ \cite{heiny2023maximum}.  
Note that if $\E[X^4]=\infty$ and consequently $\tr(\S^2)$ does not converge to the normal distribution (see Theorem~\ref{thm:non-clt}), our methods in the proof of Theorem~\ref{thm:ind} cease to work. Thus the joint limit behavior of $N_n$ and $\tr(\S^2)$ is still an open problem in the case of regularly varying $X$ with index $\alpha\in(2,4)$. In this case the dominating part of $n^2 \tr(\S^2)$ is $\sum_{i=1}^p\sum_{t=1}^nX_{it}^4$ and hence the sum of iid random variables which are regularly varying with index $\alpha/4$. In \cite{chow1978sum} it is shown that this sum is not asymptotically independent of the maximum of the $X_{it}^4$.  Therefore, we conjecture that $Z_n$ and $N_n$ will not be asymptotically independent anymore. 

As a consequence of Theorem~\ref{thm:ind} we obtain the asymptotic independence of $Z_n$ and a fixed number of upper order statistics of the random variables $(d_p(\sqrt{n}S_{ij}-d_p))_{1\leq i<j\leq p}$, which we denote by
\begin{equation}\label{eq:ordered}
G_{n,(1)}\ge G_{n,(2)}\ge \cdots \ge G_{n,(p(p-1)/2)}\,.
\end{equation}
\begin{corollary}\label{cor:2.8}
  Let $Z\sim\mathcal{N}(0,1)$ be independent of  an iid sequence $(E_i)_{i\ge 1}$ of standard exponentially distributed random variables and set $\Gamma_i:=E_1+\ldots+E_i$. 
  Under the conditions of Theorem \ref{thm:ind} and for fixed $k\ge 1$ it holds
\begin{align*}
\lim_{\nto}\P(Z_n\leq y, G_{n,(1)}\leq x_1,\ldots,G_{n,(k)}\leq x_k)= \P(Z\leq y)\P(-\log \Gamma_1\leq x_1,\ldots, -\log \Gamma_k\leq x_k), 
\end{align*}
where $y,x_1,\ldots, x_k\in\R$. 
\end{corollary}

\begin{proof}
Since $N_n(x,\infty)$ is the number of pairs $(i,j)$ with $1\le i<j\le p$, for which $(d_p(\sqrt{n}S_{ij}-d_p))\in (x,\infty)$, we get by Theorem~\ref{thm:ind} 
\begin{align*}
&\P(Z_n\leq y, G_{n,(1)}\leq x_1,\ldots,G_{n,(k)}\leq x_k)\\
&=\,\P\Big(Z_n\leq y, N_n(x_1,\infty)= 0, N_n(x_2,\infty)\le 1, \ldots,N_n(x_k,\infty)\leq k-1\Big)\\
&\to \P(Z\leq y)\,\P\Big(N(x_1,\infty)= 0, N(x_2,\infty)\le 1, \ldots,N(x_k,\infty)\leq k-1\Big)\,, \qquad \nto\,.
\end{align*}
In view of the representation $N\eid \sum_{i=1}^\infty\varepsilon_{-\log \Gamma_i}$
we obtain
\begin{align*}
   \P\Big(N(x_1,\infty)= 0,\ldots,N(x_k,\infty))\leq k-1\Big)= \P(-\log \Gamma_1\leq x_1,\ldots, -\log \Gamma_k\leq x_k),
\end{align*}
which proves the corollary.
\end{proof}

\subsection{Main challenges in the proof of Theorem~\ref{thm:ind}} \label{sec:strategy}

In this subsection we describe the key challenges in the proof of Theorem~\ref{thm:ind} and our novel technical contributions. 

The distribution of a point process $N_n$ is determined by the family of the distributions of the finite-dimensional random vectors 
$(N_n(B_1),\ldots, N_n(B_k))$ for any choice of suitable Borel sets $B_1,\ldots, B_k$; see \cite[Proposition~6.2.III]{daley:verejones:1988}.
The collection of these distributions is called the finite-dimensional distributions of $N_n$. Due to the dependence of the $(S_{ij})$ a direct analysis of the finite-dimensional distributions of $N_n$ is intractable. The same applies to the Laplace functional of $N_n$ which determines the distribution of $N_n$ completely and can be seen as a similar tool for a point process as the characteristic function for a real-valued random variable.

Fortunately, Kallenberg proved a sufficient condition for the weak convergence of a sequence of point processes $N_n$ towards $N$, which is often much easier to verify than the convergence of the finite-dimensional distributions or the Laplace functionals. More precisely, he showed that if $N$ is a simple point process (such as \eqref{eq:representationN}), then it is enough to ensure that $\E[N_n(I)]$ converges to $\E[N(I)]$ for any $I\in\mathcal{J}$ and that the probability of the event $\{N_n(U)=0\}$ converges to the probability of the event $\{N(U)=0\}$ for any $U\in\mathcal{U}$; see, for instance, \cite[p.~35, Theorem 4.7]{kallenberg:1983} or \cite[p.~233, Theorem 5.2.2]{embrechts:kluppelberg:mikosch:1997}. 
We define $\mathcal{U}$ as the set of finite unions of intervals and $\mathcal{J}$ as the set of intervals \mbox{in $\R$.} 

Therefore, instead of showing the convergence of the random vector $(N_n(B_1),\ldots,N_n(B_k))$ for any $k\geq 1$ and  $B_1,\ldots,B_k\in\mathcal{B}_N$, it is enough to prove the convergence of the probability of the occurrence of points in finite unions of intervals, which often greatly simplifies the proof. Our Theorem~\ref{thm:kal} below contains a similarly helpful tool, which will be essential for studying the joint asymptotic distribution of $Z_n$ and the point process $N_n$.
\begin{theorem}[Extension of Kallenberg's Theorem]\label{thm:kal}
Let $(Y_n)_n$ be a sequence of real-valued random variables converging in distribution to a random variable $Y$. In addition, let $N$ be a simple point process on $\R$ independent of $Y$ 
and let $(N_n)_n$ be a sequence of point processes. 
If the following two conditions
\begin{enumerate}
\item[(K1)] $\limsup\limits_{n\to\infty} \E[N_n(I)]\leq \E[N(I)]$,\quad $I\in\mathcal{J}$,\\
\item[(K2)] $\lim\limits_{n\to\infty} \P(Y_n\leq y,\,N_n(U)=0)=\P(Y\leq y)\P(N(U)=0)$,\quad $y\in\R,\,U\in\mathcal{U}$
\end{enumerate}
hold, then $N_n\cid N$ and $(N_n)_n$ and $(Y_n)_n$ are asymptotically independent.
\end{theorem}
Theorem~\ref{thm:kal} essentially shows that the sequence of random variables $Y_n$ and a sequence of point processes $N_n$ are asymptotically independent if the events $\{Y_n\leq y\}$ and $\{N_n(U)=0\}$ are asymptotically independent for any $y\in\R$ and  $U\in\mathcal{U}$. Since Theorem~\ref{thm:kal} requires mild assumptions on the real-valued random variables $Y_n$ and the point processes $N_n$, it is applicable to a wide variety of other settings. Moreover, as seen in  Corollary~\ref{cor:2.8}, it also yields asymptotic independence of the points of $N_n$ and $Y_n$.
\medskip

In Theorem \ref{thm:ind} the joint convergence of the point process $N_n$ of the entries of the sample covariance matrix $\S$ and the standardized Frobenius norm $Z_n$ of $\S$ is considered, which to the best of our knowledge has not previously been studied. We would like to mention that results about the joint convergence of a point process of dependent points and their sum are only available in the Gaussian case \cite{hu2009joint, peng2012joint, tan2011joint}, whose techniques are not applicable to non-Gaussian sequences. 

In view of Definition~\ref{def:2.5}, the main challenge in our case is to show the convergence in distribution of the random vectors $(Z_n, N_n(B_1),\ldots, N_n(B_k))$ for any $k\geq 1$ and $B_1,\ldots, B_k\in\mathcal{B}_N$. Note that every summand of $Z_n$ is dependent on a lot of points of $N_n$.
To overcome this challenge we developed a novel technical tool Theorem~\ref{thm:kal}, which allows us to reduce the convergence in distribution of the random vector above to the following two conditions: 
\begin{enumerate}
    \item[(K1')] $\limsup\limits_{n\to\infty} \E[N_n(I)]=\E[N(I)]$ for any interval $I$,
    \item[(K2')] $\lim\limits_{n\to\infty} \P(Z_n\leq y, N_n(U)=0)= \Phi(y)\P(N(U)=0)$ for any finite union of intervals $U$.
\end{enumerate}
Condition (K1') can be shown through normal approximation to large deviation probabilities. The challenging part is condition (K2'), which controls the dependence between $Z_n$ and $N_n$. 
The advantage of considering $\P(Z_n\leq y, N_n(U)=0)$, respectively $\P(Z_n\leq y, N_n(U)\neq 0)$, instead of probabilities for the vector $(Z_n, N_n(B_1),\ldots, N_n(B_k))$ is that we can write 
\begin{align*}
    \P(Z_n\leq y, N_n(U)\neq 0)=\P\Big(Z_n\leq y, \bigcup_{I\in\Lambda_n} B_I\Big)\,,
\end{align*}
where $B_I=\{d_p(\sqrt{n}S_{ij}-d_p)\in U\}$ for $I=(i,j)\in\Lambda_n=\{(i,j):1\leq i<j\leq p\}$.
Then the Bonferroni bounds yield for $k\geq 1$
\begin{equation}
\sum_{d=1}^{2k}(-1)^{d-1}W_{n,d}\leq\P\Big(Z_n\leq y, \bigcup_{I\in\Lambda_n} B_I\Big)\leq \sum_{d=1}^{2k-1}(-1)^{d-1}W_{n,d}\,,
\end{equation}
where $W_{n,d}:= \sum_{I_1<\ldots<I_d}\P\big(Z_n\leq y, \bigcap_{\ell=1}^dB_{I_\ell}\big)$  is a sum of probabilities of the intersections of finitely many $B_I$'s.
The number of summands of $Z_n$, which are dependent on $B_{I_1},\ldots,B_{I_d}$, is of order $p$. We identify these dependent summands $Z_{n,d}$ and show that they are negligible, i.e.,
\begin{align*}
    \sum_{I_1<\ldots<I_d} \P(|Z_{n,d}|>\delta)\to 0,\qquad n\to\infty.
\end{align*}
The remaining summands $Z_n-Z_{n,d}$ are independent from $B_{I_1},\ldots,B_{I_d}$ and therefore
$$\P\Big(Z_n-Z_{n,d}\leq y, \bigcap_{\ell=1}^dB_{I_\ell}\Big)=\P\big(Z_n-Z_{n,d}\leq y) \,\P\Big( \bigcap_{\ell=1}^dB_{I_\ell}\Big)\,.$$
By first letting $\nto$ and then $k\to \infty$, we can show that 
$$\P(Z_n\leq y, N_n(U)\neq 0) - \P(Z_n\leq y) \, \P( N_n(U)\neq 0) \to 0\,, \qquad \nto\,,$$
from which we deduce (K2'). The detailed proof of  Theorem~\ref{thm:ind} will be presented in Section~\ref{sec:3.4}. 

\subsection{An application to independence testing}\label{sec:test}

Consider a sample $\y_1,\ldots, \y_n$ from the $p$-dimensional population $\bfSigma^{1/2}\x$, where $\bfSigma$ is an (unknown) non-random positive definite $p\times p$ matrix and  $\x$ has iid components with mean zero and variance one. 
The largest off-diagonal entry of the sample covariance matrix $\S=(S_{ij})=\frac{1}{n}\sum_{t=1}^n\y_t \y_t^{\top}$
is a popular statistic for structural tests about properties of $\bfSigma$; we refer to the review paper \cite{cai:2017} for an extensive summary and detailed references.   We are interested in the null hypothesis of independence $H_0$:$\,\bfSigma=\bfI_p$.
In what follows we will present a rather simplistic extension of the classical maximum-type tests as an application of Theorem~\ref{thm:ind}. This application is by no means perfect, it rather serves as an illustration of the potential of the asymptotic independence derived in Theorem~\ref{thm:ind} regarding statistical tests\footnote{A thorough analysis will be topic of future research by the authors.}. 

Under the null hypothesis Theorem~\ref{thm:ind} studies the joint limiting distribution of $(Z_n,N_n)$, where $Z_n$ is given in \eqref{clt} and
$$N_n=\sum_{1\leq i<j\leq p}\varepsilon_{d_p(\sqrt{n}S_{ij}-d_p)}\,.$$
From \eqref{eq:ordered} recall the definition of 
$G_{n,(1)}\ge \cdots \ge G_{n,(p(p-1)/2)}$.
Assuming the conditions of Theorem~\ref{thm:ind}, the asymptotic independence of $Z_n$ and $N_n$ implies for fixed $k\ge 1$ (c.f.\ Corollary~\ref{cor:2.8}) that
\begin{equation}\label{eq:corimp}
\big(Z_n, G_{n,(1)},\ldots,G_{n,(k)}\big) \cid 
\big( Z, -\log \Gamma_1,\ldots, -\log \Gamma_k \big)\,, \qquad \nto\,,
\end{equation}
where $Z\sim\mathcal{N}(0,1)$ is independent of the iid sequence $(E_i)_{i\ge 1}$ of standard exponentially distributed random variables and $\Gamma_i:=E_1+\ldots+E_i$.   
Next we introduce a variety of different statistics:
\begin{equation}
\begin{alignedat}{2}
    T_{1}&:= G_{n,(1)},\quad
    && \text{(rescaled) largest off-diagonal entry of $\bfS$,}\\
    T_{2,k}&:=G_{n,(1)}-G_{n,(k)},
    && \text{gap of largest and $k$-th largest entry,}\\
    T_{3,k}&:=\max_{1\leq i\leq k-1} \big(G_{n,(i)}-G_{n,(i+1)}\big),\quad
    && \text{maximum spacing between consecutive order statistics,}\\
    T_{4,k}&:=\sum_{i=1}^{k-1}(G_{n,(i)}-G_{n,(i+1)})^2,
    && \text{sum of squared spacings between consec.\ order statistics.}
\end{alignedat}
\end{equation}
For the first statistic it holds that $T_{1}\cid -\log \Gamma_1$, which is standard Gumbel distributed with distribution function $\Lambda(x)=\exp(-\e^{-x})$. Recall the well--known fact that 
\begin{align*}
    \Big(\frac{\Gamma_1}{\Gamma_{k+1}},\ldots, \frac{\Gamma_k}{\Gamma_{k+1}} \Big)\stackrel{d}{=}(U_{(k)},\ldots,U_{(1)}),
\end{align*}
where the right-hand vector consists of the order statistics of $k$ iid uniform random variables on $[0,1]$. In combination with \eqref{eq:corimp} we get for the other statistics, as $\nto$,
\begin{align*}
    T_{2,k}&\cid \log(\Gamma_k/\Gamma_1)\eid \log(U_{(1)}/U_{(k)})\,,\\
    T_{3,k}&\cid \max_{1\leq i\leq k-1} \log (\Gamma_{i+1}/\Gamma_i)\eid \max_{1\leq i\leq k-1} \log(U_{(k-i)}/U_{(k-i+1)})\,,\\
    T_{4,k}&\cid \sum_{i=1}^{k-1} (\log (\Gamma_{i+1}/\Gamma_i))^2\eid \sum_{i=1}^{k-1}(\log(U_{(k-i)}/U_{(k-i+1)}))^2\,.
\end{align*}
Next we construct the random variables
\begin{align*}
    P_{T_{1}}&:=1-\Lambda(T_{1,k})\,,\\ P_{T_{i,k}}&:=1-F_{i,k}(T_{i,k})\,, \qquad  i=2,3,4,\\
    P_{Z_n}&:=1-\Phi(Z_n)\,,
\end{align*}
where $F_{2,k}$, $F_{3,k}$, $F_{4,k}$ are the distribution functions of $\log(U_{(1)}/U_{(k)})$, $\max_{1\leq i\leq k-1} \log(U_{(k-i)}/U_{(k-i+1)})$ and $\sum_{i=1}^{k-1}(\log(U_{(k-i)}/U_{(k-i+1)}))^2$, respectively, and $\Phi$ denotes the standard normal distribution function. 
By \eqref{eq:corimp}, $P_{Z_n}$ is asymptotically independent of $P_{T_1}, P_{T_{2,k}}, P_{T_{3,k}}, P_{T_{4,k}}$. Additionally, each of these random variables converges to the uniform distribution on $[0,1]$. 

We propose the following four test statistics 
\begin{equation*}
    \mathcal{T}_{1,n}:=\min\{P_{Z_n},P_{T_{1}}\} \quad \text{ and } \quad 
    \mathcal{T}_{i,n}:=\mathcal{T}^{(k)}_{i,n}:=\min\{P_{Z_n},P_{T_{i,k}}\}\quad \text{ for } i\in\{2,3,4\}\,.
\end{equation*}
The null hypothesis $H_0$ is rejected by test $i\in\{1,\ldots,4\}$, whenever 
 \begin{align}	 \label{hd5}	
\mathcal{T}_{i,n}<1-\sqrt{1-\beta}\,.
	\end{align} 
The next result establishes the asymptotic distribution of $\mathcal{T}_{i,n}$ from which we deduce that the tests in \eqref{hd5} have asymptotic level $\beta\in (0,1)$.
\begin{corollary}\label{cor:2.10}
    Under the conditions of Theorem \ref{thm:ind}, it holds that
    $$\mathcal{T}_{i,n} \cid \min(U,V)\,, \qquad \nto\,, i=1,\ldots,4\,,$$
    where $U$ and $V$ are independent random variables uniformly distributed on $[0,1]$. 
\end{corollary}
Since $\P(\min\{U,V\}\le x)=2x-x^2$, $x\in[0,1]$, it follows for $\beta\in (0,1)$
   \begin{align*}
    \lim_{\nto}\P(\mathcal{T}_{i,n}<1-\sqrt{1-\beta})&= \P(\min\{U,V\}<1-\sqrt{1-\beta})\\
    &=2(1-\sqrt{1-\beta})-(1-\sqrt{1-\beta})^2=\beta\,.
\end{align*}

\section{Proofs}
\label{sec:4}
{\it To simplify the notation in the proofs we will write $a_n\sim b_n$ for real-valued sequences $a_n$ and $b_n$ if $\lim_{\nto} a_n/b_n=1$,  $a_n \gg b_n$ if $\lim_{\nto} b_n/a_n=0$,  $a_n \lesssim b_n$ if $\limsup_{\nto} a_n/b_n=C'$ for some constant $C'\in [0,\infty)$ and
$a_n \asymp b_n$ if $a_n \lesssim b_n$ and $b_n \lesssim a_n$. Additionally, throughout the proofs $C$ denotes a positive constant which may vary from line to line.
Unless explicitly stated otherwise, all limits are for $\nto$.}

\subsection{Proof of Theorem \ref{thm:clt}}
First, notice that the magnitude of $\sigma_n$, defined in \eqref{eq:defsigma}, might be different in the cases $\E[X^4]\neq 1$ and $\E[X^4]=1$. If $\E[X^4]=1$, Markov's inequality yields for $\varepsilon>0$ that 
\begin{align*}
    \P(|X^2-1|>\varepsilon)\leq \varepsilon^{-2}\E[(X^2-1)^2]=0,
\end{align*} 
such that $X^2=1$ almost surely, which due to $\E[X]=0$ is only possible if $X$ follows a symmetric Bernoulli distribution, i.e., $\P(X=-1)=\P(X=1)=1/2$. We will often consider the case $\E[X^4]=1$ separately in the course of this proof. \medskip

To prove the statement of Theorem \ref{thm:clt}, we use a central limit theorem for martingale differences. As we need the existence of higher moments, we truncate the random variables $X_{it}$ for $i=1,\ldots, p$ and $t=1,\ldots, n$ in an appropriate way.
Let $(\beta_n)_n$ be a positive sequence, which tends to zero and suffices $\beta_n\gg(\E[|X|^4\mathds{1}_{\{|X|>\beta_n(np)^{1/4}\}}])^{1/4}$. Notice, that such a sequence exists since we can choose a sequence $\beta_n'$, which tends zero sufficiently slowly such that $\beta_n'(np)^{1/4}\to\infty$. Now we set $\beta_n'':=(\E[|X|^4\mathds{1}_{\{|X|>\beta_n'(np)^{1/4}\}}])^{1/4}$, which also tends to zero as $n\to\infty$, and choose $\beta_n\gg\max\{\beta_n',\beta_n''\}$.

For this sequence $(\beta_n)_n$, we set
\begin{align}
\bar{X}_{it}&:= X_{it}\mathds{1}_{\{|X_{it}|\leq \beta_n (np)^{1/4}\}},\qquad 1\leq i\leq p,\,1\leq t\leq n,\nonumber \\
T_{ij}&:=\sum_{t=1}^n X_{it}X_{jt}\quad \text{ and } \quad  \bar{T}_{ij}:=\sum_{t=1}^n\bar{X}_{it}\bar{X}_{jt}, \qquad 1\leq i,\,j\leq p.\label{bart}
\end{align}
Since $\E[X^4]<\infty$ an application of Lemma \ref{lem:cut} with $s=4$ yields that it suffices to verify \eqref{clt} with $Z_n$ replaced by
\begin{align}\label{cutsum}
\bar{Z}_n=\frac{1}{n^2\sigma_n}\sum_{1\leq i,\,j\leq p} \big(\bar{T}_{ij}^2-\E[\bar{T}_{ij}^2]\big)\,.
\end{align}
It will be convenient to write $\bar{Z}_n$ as a sum of  martingale differences. In what follows, the notation 
$$\tx_i=(X_{i1}, \ldots, X_{in})\,, \qquad  i=1,\ldots,p$$
will be helpful. Noting that $\bar{T}_{ij}=\bar{T}_{ji}$, we start by writing $\bar{Z}_n$ as 
\begin{align*}
\bar{Z}_n&=\frac{2}{n^2\sigma_n}\sum_{j=1}^p\sum_{i=1}^{j-1}(\bar{T}_{ij}^2-\E[\bar{T}_{ij}^2|\tx_i])+\frac{2}{n^2\sigma_n}\sum_{j=1}^p\sum_{i=1}^{j-1}(\E[\bar{T}_{ij}^2|\tx_i]-\E[\bar{T}_{ij}^2])\\
&\quad +\frac{1}{n^2\sigma_n}\sum_{j=1}^p(\bar{T}_{jj}^2-\E[\bar{T}_{jj}^2]).
\end{align*} 
Using $\E[\bar{T}_{ij}^2|\tx_i]=\sum_{t=1}^n\sum_{u=1}^n\bar{X}_{it}\bar{X}_{iu}\E[\bar{X}_{jt}\bar{X}_{ju}]=\E[\bar{T}_{i,i+1}^2| \tx_i]$ for $j=i+1,\ldots,p$ we have
\begin{align*}
&\sum_{j=1}^p\sum_{i=1}^{j-1}(\E[\bar{T}_{ij}^2|\tx_i]-\E[\bar{T}_{ij}^2])
=\sum_{i=1}^{p-1} (p-i)(\E[\bar{T}_{i,i+1}^2|\tx_i]-\E[\bar{T}_{12}^2]).
\end{align*}
Setting
\begin{align*}
M_j:=2(p-j)(\E[\bar{T}_{j,j+1}^2|\tx_j]-\E[\bar{T}_{12}^2])+ 2\sum_{i=1}^{j-1}(\bar{T}_{ij}^2-\E[\bar{T}_{ij}^2|\tx_i])+(\bar{T}_{jj}^2-\E[\bar{T}_{jj}^2])
\end{align*}
we get
\begin{align*}
\bar{Z}_n=\frac{1}{n^2\sigma_n}\sum_{j=1}^p M_j,
\end{align*}
where $(M_j)_{j\ge 1}$ is a martingale difference sequence with respect to the filtration $(\mathcal{F}_{j})_{j\ge 0}$, where $\mathcal{F}_{j}$ is the sigma algebra generated by $\{\tx_1,\ldots,\tx_j\}$. Indeed, we have $\E[M_j|\mathcal{F}_{j-1}]=0$. 

By the Lindeberg-Feller theorem for martingales (see, for example, \cite[Theorem 8.2.4, p.\,344]{durrett2019probability}), the convergence $\bar{Z}_n\cid \mathcal{N}(0,1)$ is implied by the following two assertions:
\begin{enumerate}
\item $A_n:=\frac{1}{n^4\sigma_n^2}\sum_{j=1}^p\E[M_j^2|\mathcal{F}_{j-1}]\cip 1\,, \quad \nto\,,$
\item $\frac{1}{n^8\sigma_n^4}\sum_{j=1}^p\E[M_j^4|\mathcal{F}_{j-1}]\cip 0\,, \quad \nto\,.$
\end{enumerate}
\subsubsection*{Proof of (1).} To prove (1) we will show $\E[A_n]\to 1$ and $\Var(A_n)\to 0$ as $n\to\infty$. Since $(M_j)$ is a martingale difference sequence, we have $\E[A_n]=\Var(\bar{Z}_n)$ and thus we get
{\small
\begin{align}
&\E[A_n]=\Var\Big(\frac{1}{n^2\sigma_n}\sum_{i,j=1}^p\bar{T}_{ij}^2\Big)\\
&=\frac{1}{n^4\sigma_n^2}\Big(\Var\Big(\sum_{i=1}^p\bar{T}_{ii}^2\Big)+4 \Var\Big(\sum_{1\leq i<j\leq p}\bar{T}_{ij}^2\Big)+4\sum_{k=1}^{p-1}\sum_{j=k+1}^p\sum_{i=1}^p\operatorname{Cov}(\bar{T}_{ii}^2,\bar{T}_{kj}^2)\Big)\label{Var}\\
&=\frac{p}{n^4\sigma_n^2}\Var(\bar{T}_{11}^2)+\frac{2p(p-1)}{n^4\sigma_n^2}\Var(\bar{T}_{12}^2)+\frac{4p(p-1)^2}{n^4\sigma_n^2}\Cov(\bar{T}_{12}^2,\bar{T}_{13}^2)+\frac{4p(p-1)}{n^4\sigma_n^2}\Cov(\bar{T}_{11}^2,\bar{T}_{12}^2),
\end{align}}
where we used 
\begin{align*}
\Var\Big(\sum_{1\leq i<j\leq p}\bar{T}_{ij}^2\Big)&=\sum_{1\leq i<j\leq p}\Var(\bar{T}_{ij}^2)+\sum_{1\leq i<j\leq p}\sum_{\substack{1\leq k<l\leq p\\ (k,l)\neq (i,j)}}\operatorname{Cov}(\bar{T}_{ij}^2,\bar{T}_{kl}^2)
\end{align*}
and that $\Cov(\bar{T}_{ij}^2,\bar{T}_{kl}^2)=0$ if $\{i,j\}\cap\{k,l\}=\emptyset$. In the case $\E[X^4] \neq 1$ we then obtain by Lemma \ref{lem:var} that
\begin{align*}
\E[A_n]=\frac{1}{\sigma_n^2}\Big(4\frac{p^2}{n^2}+4(\E[X^4]-1)\Big(\frac{p^3}{n^3}+2\frac{p^2}{n^2}+\frac{p}{n}\Big)+o\Big(\frac{p^3}{n^3}+\frac{p}{n}\Big)\Big)=1+\frac{o\big(\frac{p^3}{n^3}+\frac{p}{n}\big)}{\sigma_n^2}  \to 1,
\end{align*}
since $\sigma_n^2\asymp p/n +(p/n)^3$. If $\E[X^4]= 1$ one can similarly check that $\E[A_n]\to 1$.

Now we turn to $\Var(A_n)$. To this end, we need the moments of the truncated random variable $\bar{X}$ and note that
\begin{align}
|\E[\bar{X}]|\leq \E[|X|\mathds{1}_{\{|X|>(np)^{1/4}\beta_n\}}]\leq \frac{1}{(np)^{3/4}\beta_n^{3}}\E[X^4\mathds{1}_{\{|X|>(np)^{1/4}\beta_n\}}]= o((np)^{-3/4})\,. \label{mo1}
\end{align}
Similarly we obtain
\begin{align}
\E[\bar{X}^2]=1+o((np)^{-1/2}) \quad \text{ and } \quad \E[\bar{X}^4]=\E[X^4]+o(1)\,,\label{mo2}
\end{align}
and for higher moments we get 
\begin{align}
\E[|\bar{X}|^{4+k}]\leq \E[\bar{X}^4](np)^{k/4}\beta_n^k=o((np)^{k/4})\,, \qquad k\ge 1\,.\label{mo3}
\end{align}

From the definitions of $M_j$ and the sigma algebra $\mathcal{F}_{j-1}$, we deduce that there exists some constant $C'_n$ only depending on $n$ such that
\begin{align}
 \E[M_j^2|\mathcal{F}_{j-1}]&=C'_n+ 8(p-j)\E\Big[\sum_{i=1}^{j-1}(\bar{T}_{ij}^2-\E[\bar{T}_{ij}^2|\tx_i])(\E[\bar{T}_{j,j+1}^2|\tx_j]-\E[\bar{T}_{12}^2])\Big|\mathcal{F}_{j-1}\Big] \\
 &\quad +4\,\E\Big[\sum_{i=1}^{j-1}(\bar{T}_{ij}^2-\E[\bar{T}_{ij}^2|\tx_i])(\bar{T}_{jj}^2-\E[\bar{T}_{11}^2])\Big|\mathcal{F}_{j-1}\Big]\\
 &\quad +4\,\E\Big[\Big(\sum_{i=1}^{j-1}(\bar{T}_{ij}^2-\E[\bar{T}_{ij}^2|\mathcal{F}_{j-1}])\Big)^2\Big|\mathcal{F}_{j-1}\Big].\label{Mqu}
\end{align}
The second term can be written as
\begin{align*}
    8(p-j)\sum_{i=1}^{j-1}\sum_{t_1,t_2=1}^n\underbrace{\E\big[(\bar{X}_{jt_1}\bar{X}_{jt_2}-\E[\bar{X}_{jt_1}\bar{X}_{jt_2}])(\E[\bar{T}_{j,j+1}^2|\tx_j]-\E[\bar{T}_{12}^2])\big]}_{=:K_{t_1,t_2,1}}\bar{X}_{it_1}\bar{X}_{it_2}\,,
\end{align*}
where $K_{t_1,t_2,1}$ can be expressed as
\begin{align*}
    K_{t_1,t_2,1}=\sum_{t_3,t_4=1}^n \E[(\bar{X}_{jt_1}\bar{X}_{jt_2}-\E[\bar{X}_{1t_1}\bar{X}_{1t_2}])(\bar{X}_{jt_3}\bar{X}_{jt_4}-\E[\bar{X}_{1t_3}\bar{X}_{1t_4}])]\E[\bar{X}_{2t_3}\bar{X}_{2t_4}].
\end{align*}
Notice that $K_{t_1,t_2,1}=O(1)$ as $n\to\infty$ since the summands are $0$ if $\{t_1,t_2\}\cap\{t_3,t_4\}=\emptyset$. Otherwise if $t_3\neq t_4$ we have $\E[\bar{X}_{2t_3}\bar{X}_{2t_4}]=\E[\bar{X}]^2=o((np)^{-3/2})$. The means are bounded in every case.
With the third term of \eqref{Mqu} we proceed similarly
\begin{align*}
     &4\E\Big[\sum_{i=1}^{j-1}(\bar{T}_{ij}^2-\E[\bar{T}_{ij}^2|\tx_i])(\bar{T}_{jj}^2-\E[\bar{T}_{11}^2])\Big|\mathcal{F}_{j-1}\Big]\\
     &=4\sum_{i=1}^{j-1}\sum_{t_1,t_2=1}^n\underbrace{\E[(\bar{X}_{jt_1}\bar{X}_{jt_2}-\E[\bar{X}_{1t_1}\bar{X}_{1t_2}])(\bar{T}_{jj}^2-\E[\bar{T}_{11}^2])]}_{K_{t_1,t_2,2}}\bar{X}_{it_1}\bar{X}_{it_2}\,,
\end{align*}
where $K_{t_1,t_2,2}$ can be written as
\begin{align*}
    K_{t_1,t_2,2}=\sum_{t_3,t_4=1}^n\E[(\bar{X}_{jt_1}\bar{X}_{jt_2}-\E[\bar{X}_{1t_1}\bar{X}_{1t_2}])(\bar{X}^2_{jt_3}\bar{X}^2_{jt_4}-\E[\bar{X}^2_{1t_3}\bar{X}^2_{1t_4}])].
\end{align*}
As $n\to\infty$ it holds that $K_{t_1,t_2,2}=O(n+(np)^{1/2})$ because the expectation is zero if $\{t_1,t_2\}\cap\{t_3,t_4\}=\emptyset$, bounded if $t_3\neq t_4$ and $C\,\E[\bar{X}^6]=o((np))^{1/2}$ if $t_1=t_2=t_3=t_4$.
Hence, we are able to write $A_n$ as 
\begin{align*}
    A_n &=\frac{1}{n^4\sigma_n^2}\sum_{j=1}^p\E[M_j^2|\mathcal{F}_{j-1}]=C_n+\frac{1}{n^4\sigma_n^2}\sum_{j=1}^p\sum_{i=1}^{j-1}\sum_{t_1,t_2=1}^n K_{t_1,t_2}\bar{X}_{it_1}\bar{X}_{it_2}\\
    &\quad +\frac{4}{n^4\sigma_n^2}\sum_{j=1}^p\E\Big[\Big(\sum_{i=1}^{j-1}(\bar{T}_{ij}^2-\E[\bar{T}_{ij}^2|\mathcal{F}_{j-1}])\Big)^2\Big|\mathcal{F}_{j-1}\Big]\,,
\end{align*}
    where $C_n$ is a constant only depending on $n$ and on the distribution of $X$ and $K_{t_1,t_2}=8(p-j)K_{t_1,t_2,1}+4K_{t_1,t_2,2}= O(n+p)$.
The expectation of the last term can be written as
\begin{align*}
\E\Big[\Big(\sum_{i=1}^{j-1}\big(\bar{T}_{ij}^2-\E[\bar{T}_{ij}^2|\mathcal{F}_{j-1}]\big)\Big)^2\Big|\mathcal{F}_{j-1}\Big]
&= \E\Big[\Big(\sum_{i=1}^{j-1}\bar{T}_{ij}^2\Big)^2\big|\mathcal{F}_{j-1}\Big]-\E\Big[\sum_{i=1}^{j-1}\bar{T}_{ij}^2\big|\mathcal{F}_{j-1}\Big]^2\\
&=\sum_{i_1,i_2=1}^{j-1}(\E[\bar{T}_{i_1j}^2\bar{T}_{i_2j}^2|\mathcal{F}_{j-1}]-\E[\bar{T}_{i_1j}^2|\mathcal{F}_{j-1}]\E[\bar{T}_{i_2j}^2|\mathcal{F}_{j-1}]).
\end{align*}
Using Lemma \ref{lem:zw} we see that, for sequences $C_{n,1}, C_{n,2}$ and $C_{n,3}$ tending to constants,  
\begin{align*}
    A_n-C_n&=\frac{1}{n^4\sigma_n^2}\sum_{j=1}^p\sum_{i=1}^{j-1}\sum_{t_1,t_2=1}^n K_{t_1,t_2}\bar{X}_{it_1}\bar{X}_{it_2}\\
    &\quad+(\E[\bar{X}^4]-(\E[\bar{X}^2])^2)\frac{4}{n^4\sigma_n^2}\sum_{j=1}^p\sum_{i_1,i_2=1}^{j-1}\sum_{t=1}^n\bar{X}_{i_1t}^2\bar{X}_{i_2t}^2\\
    &\quad+\frac{1}{n^4\sigma_n^2}\sum_{j=1}^p\sum_{i_1,i_2=1}^{j-1}\sum_{\substack{t_1,t_2=1\\t_1\neq t_2}}^n \Big(\E[\bar{X}]C_{n,1}\bar{X}_{i_1t_1}^2\bar{X}_{i_2t_1}\bar{X}_{i_2t_2}+C_{n,2}\bar{X}_{i_1t_1}\bar{X}_{i_1t_2}\bar{X}_{i_2t_1}\bar{X}_{i_2t_2}\Big)\\
    &\quad +\E[\bar{X}]^2\frac{C_{n,3}}{n^4\sigma_n^2}\sum_{j=1}^p\sum_{i_1,i_2=1}^{j-1}\sum_{\substack{t_1,t_2,t_3=1\\|\{t_1,t_2,t_3\}|=3}}^n\bar{X}_{i_1t_1}\bar{X}_{i_1t_2}\bar{X}_{i_2t_1}\bar{X}_{i_2t_3}\\
    &=:\xi_{n,1}+\xi_{n,2}+\xi_{n,3}+\xi_{n,4}+\xi_{n,5}.
\end{align*}
In view of
\begin{align*}
\Var(A_n)&=\Var(\xi_{n,1}+\xi_{n,2}+\xi_{n,3}+\xi_{n,4}+\xi_{n,5})\\
&\lesssim \Var(\xi_{n,1})+\Var(\xi_{n,2})+\Var(\xi_{n,3})+\Var(\xi_{n,4})+\Var(\xi_{n,5}),
\end{align*}
it suffices to show that the variances of $\xi_{n,1}$, $\xi_{n,2}$, $\xi_{n,3}$, $\xi_{n,4}$ and $\xi_{n,5}$ tend to zero as $n\to\infty$. 

We start by bounding the variance of $\xi_{n,1}$:
\begin{align*}
    \Var(\xi_{n,1})
    &\lesssim\frac{p^3}{n^8\sigma_n^4}\sum_{t_1, t_2,t_3,t_4=1}^n \big| K_{t_1,t_2}K_{t_3,t_4}(\E[\bar{X}_{1t_1}\bar{X}_{1t_2}\bar{X}_{1t_3}\bar{X}_{1t_4}]-\E[\bar{X}_{1t_1}\bar{X}_{1t_2}]\E[\bar{X}_{1t_3}\bar{X}_{1t_4}])\big|\,.
\end{align*}
Since the summands are zero if $|\{t_1,t_2,t_3,t_4\}|=4$, bounded by $\E[\bar{X}]$, which tends to zero as $n\to\infty$ if $|\{t_1,t_2,t_3,t_4\}|=3$ and bounded above by a constant else, we get
\begin{align*}
    \Var(\xi_{n,1})=o\Big(\frac{p^3(n+p)^2}{n^5\sigma_n^4}\Big).
\end{align*}

For the variance of $\xi_{n,2}$ it holds that 
\begin{align*}
    \Var(\xi_{n,2})&\lesssim \frac{p}{n^7\sigma_n^4}\sum_{j=1}^p\Var\bigg(\sum_{i_1,i_2=1}^{j-1}\bar{X}_{i_11}^2\bar{X}_{i_21}^2\bigg)\\
    &\lesssim\frac{p}{n^7\sigma_n^4}\sum_{j=1}^p\sum_{i_1,i_2,i_3=1}^{j-1}\Cov(\bar{X}_{i_11}^2\bar{X}_{i_21}^2,\bar{X}_{i_11}^2\bar{X}_{i_31}^2)\,.
\end{align*}
As the covariance above is equal to $\E[\bar{X}^8]-(\E[{X}^4])^2=o(np)$ if $i_1=i_2=i_3$, $\E[\bar{X}^6]\E[\bar{X}^2]-\E[\bar{X}^4](\E[\bar{X}^2])^2=o((np)^{1/2})$ if $|\{i_1,i_2,i_3\}|=2$  and bounded above by a constant if $|\{i_1,i_2,i_3\}|=3$, we find
\begin{align*}
\Var(\xi_{n,2})=o\Big(\frac{p^4}{n^6\sigma_n^4}+\frac{p^5}{n^7\sigma_n^4}\Big).
\end{align*}
 By similar arguments, we obtain
\begin{align*}
\Var(\xi_{n,3})\lesssim \frac{\E[\bar{X}]^2p}{n^8\sigma_n^4}\sum_{j=1}^p\sum_{i_1,i_2,i_3,i_4=1}^{j-1}\sum_{\substack{t_1,t_2,t_3,t_4=1\\t_1\neq t_2,\,t_3\neq t_4}}^n\Cov(\bar{X}_{i_1t_1}^2\bar{X}_{i_2t_1}\bar{X}_{i_2t_2},\bar{X}_{i_3t_3}^2\bar{X}_{i_4t_3}\bar{X}_{i_4t_4})\,,
\end{align*}
where the covariance above is $\E[\bar{X}^6]\E[\bar{X}^2]-\E[\bar{X}^3]^2\E[\bar{X}]^2=o((np)^{1/2})$ if $i_1=i_2=i_3=i_4$ and $t_1=t_3$ and $t_2=t_4$, zero if $|\{i_1,\ldots,i_4\}|=4$ or $|\{t_1,\ldots,t_4\}|=4$ and bounded by a constant in the remaining cases. Therefore, we get using $\E[\bar{X}]^2=o((np)^{-3/2})$ that
\begin{align*}
    \Var(\xi_{n,3})=o\bigg(\frac{p^{7/2}}{n^{13/2}\sigma_n^4}\bigg).
\end{align*}
For the variance of $\xi_{n,4}$ we have
\begin{align*}
    \Var(\xi_{n,4})\lesssim \frac{p}{n^8\sigma_n^4}\sum_{j=1}^p\sum_{i_1,i_2,i_3,i_4=1}^{j-1}\sum_{\substack{t_1,t_2,t_3,t_4=1\\t_1\neq t_2,\,t_3\neq t_4}}^n\Cov(\bar{X}_{i_1t_1}\bar{X}_{i_1t_2}\bar{X}_{i_2t_1}\bar{X}_{i_2t_2},\bar{X}_{i_3t_3}\bar{X}_{i_3t_4}\bar{X}_{i_4t_3}\bar{X}_{i_4t_4})\,.
\end{align*}
Again the covariance above is zero, if $|\{i_1,\ldots,i_4\}|=4$ or $|\{t_1,\ldots,t_4\}|=4$. If $|\{i_1,\ldots,i_4\}|=3$ it is bounded by $\E[\bar{X}^2]^2\E[\bar{X}]^4-\E[\bar{X}]^8=o((np)^{-3})$ and by a constant in the remaining cases. It follows $\Var(\xi_{n,4})=O\big(\frac{p^4}{n^5\sigma_n^4}\big)$.\\
Next, since $\Var(\bar{X}_{i_1t_1}\bar{X}_{i_1t_2}\bar{X}_{i_2t_1}\bar{X}_{i_2t_3})\lesssim 1$ for $|\{t_1,\ldots,t_4\}|=3$ and $\E[\bar{X}]^4=o((np)^{-3})$ we obtain $\Var(\xi_{n,5})=o(p^3n^{-5}\sigma_n^{-4})$.

Finally, we combine our variance estimates. In the case $\E[X^4]\neq 1$, it holds $\sigma_n^4\asymp (p/n)^2 +(p/n)^6$ which implies
\begin{equation}\label{eq:sesgss}
 \lim_{\nto} \max_{i=1,\ldots,5}   \Var(\xi_{n,i})=0\,. 
\end{equation}
  In the Bernoulli case $\E[X^4]= 1$, we have $\sigma_n^4\asymp (p/n)^4$. As $K_{t_1,t_2,1}$ and $K_{t_1,t_2,2}$ are zero in this case, $\xi_{n,1}=0$. Since $\E[\bar{X}]=\E[X]=0$ and $\E[\bar{X}^2]^2=\E[\bar{X}^4]=1$ one has $\xi_{n,2}=\xi_{n,3}=\xi_{n,5}=0$. Repeating the above considerations for $\xi_{n,4}$, one can show that $\Var(\xi_{n,4})\to 0$ as $n\to\infty$, establishing \eqref{eq:sesgss} in the Bernoulli case as well.
  
Equation \eqref{eq:sesgss} concludes the proof of $\Var(A_n)\to 0$, as $n\to\infty$. In combination with $\E[A_n]\to 1$, this proves the desired $A_n\cip 1$.

\subsubsection*{Proof of (2).}
We write
\begin{align*}
M_j&=2(p-j)(\E[\bar{T}_{j,j+1}^2|\tx_j]-\E[\bar{T}_{12}^2])+ 2\sum_{i=1}^{j-1}(\bar{T}_{ij}^2-\E[\bar{T}_{ij}^2|\tx_j])\\
&-2\sum_{i=1}^{j-1}(\E[\bar{T}_{ij}^2|\tx_i]-\E[\bar{T}_{12}^2])+2\sum_{i=1}^{j-1}(\E[\bar{T}_{ij}^2|\tx_j]-\E[\bar{T}_{12}^2])+(\bar{T}_{jj}^2-\E[\bar{T}_{11}^2])\\
&=: M_{j1}+M_{j2}+M_{j3}+M_{j4}+M_{j5}
\end{align*}
and bound the fourth moment of $M_j$ in the following way
\begin{align*}
\E[M_j^4]\lesssim \E[(M_{j1}+M_{j4})^4]+\E[M_{j2}^4]+\E[M_{j3}^4]+\E[M_{j5}^4].
\end{align*}
First we take a look at 
\begin{align*}
&\E[(\E[\bar{T}_{j,j+1}^2|\tx_j]-\E[\bar{T}_{12}^2])^4]\\
&=\E\Big[\Big((1+o((np)^{-1/2})(\bar{T}_{jj}-\E[\bar{T}_{11}])+o((np)^{-3/2})\sum_{t,u=1}^n(\bar{X}_{jt}\bar{X}_{ju}-\E[\bar{X}_{1t}\bar{X}_{1u}])\Big)^4\Big]\\
&\lesssim \E[(\bar{T}_{jj}-\E[\bar{T}_{11}])^4]+o(1)\,.
\end{align*}
For the last mean we have
\begin{align*}
\E[(\bar{T}_{jj}-\E[\bar{T}_{11}])^4]&=\E\Big[\Big(\sum_{t=1}^n(\bar{X}_{jt}^2-\E[\bar{X}^2])\Big)^4\Big]\\
&\lesssim n\E[\bar{X}^8]+n^2\E[\bar{X}^4]^2=o(n^2p),
\end{align*}
where we used \eqref{mo2} and \eqref{mo3} in the last line. Hence, we get after simplifying $M_{j4}$
\begin{align*}
\E[(M_{j1}+M_{j4})^4]=16(p-1)^4\E[(\E[\bar{T}_{j,j+1}^2|\tx_j]-\E[\bar{T}_{12}^2])^4]=o(p^5n^2)
\end{align*}
and by the Marcinkiewicz-Zygmund inequality, see for example \cite[Theorem 2, p.386]{teicher1978probability} in combination with Hölder inequality (see also \cite[Lemma 2, p.24]{feng2022max}), it follows 
\begin{align*}
\E[M_{j3}^4]\lesssim (j-1)\sum_{i=1}^{j-1}\E[(\E[\bar{T}_{ij}^2|\tx_i]-\E[\bar{T}_{12}^2])^4]=(j-1)^2o(n^2p).
\end{align*}
To bound the fourth moment of $M_{j2}$ we consider the eighth moment of $\bar{T}_{12}$
\begin{align*}
\E[\bar{T}_{12}^8]&=\E\Big[\Big(\sum_{t=1}^n\bar{X}_{1t}\bar{X}_{2t}\Big)^8\Big]\\
&\lesssim n(\E[\bar{X}^8])^2+n^2(\E[\bar{X}^6]^2\E[\bar{X}^2]^2+\E[\bar{X}^7]^2\E[\bar{X}]^2+\E[\bar{X}^5]^2\E[\bar{X}^3]^2+\E[\bar{X}^4]^4)\\
&+n^3+n^4+n^5\E[\bar{X}]^4+n^6\E[\bar{X}]^8+n^7\E[\bar{X}]^{12}+n^8\E[\bar{X}]^{16}=O(n^4)+o(n^3p^2),
\end{align*}
where we used \eqref{mo1}, \eqref{mo2} and \eqref{mo3} in the last step. By the Marcinkiewicz-Zygmund inequality (see for instance \cite[Theorem 2, p.386]{teicher1978probability}) and Lemma 1 of \cite{feng2022max} we conclude that
\begin{align*}
\E[M_{j2}^4]&=2^4\E\Big[\E\Big[\Big(\sum_{i=1}^{j-1}(\bar{T}_{ij}^2-\E[\bar{T}_{ij}^2|\tx_j])\Big)^4\Big|\tx_j\Big]\Big]\\
&\lesssim  (j-1)\sum_{i=1}^{j-1}\E[\E[(\bar{T}_{ij}^2-\E[\bar{T}_{ij}^2|\tx_j])^4|\tx_j]]\\
&\leq  (j-1)^2\E[\bar{T}_{12}^8]=(j-1)^2(O(n^4)+o(n^3p^2)).
\end{align*}
Finally, it holds for the fourth moment of $M_{j5}$ that
\begin{align*}
\E[M_{j5}^4]=\sum_{t_1,\ldots, t_8=1}^n \E[(\bar{X}_{jt_1}^2\bar{X}_{jt_2}^2-\E[\bar{X}_{jt_1}^2\bar{X}_{jt_2}^2])\cdots (\bar{X}_{jt_7}^2\bar{X}_{jt_8}^2-\E[\bar{X}_{jt_7}^2\bar{X}_{jt_8}^2])]
\end{align*}
Observe that if there are more than five different indices $t_1,\ldots,t_8$ then one factor in the mean above is independent from the rest and the mean is zero, thus by \eqref{mo1}, \eqref{mo2} and \eqref{mo3}
\begin{align*}
&\E[M_{j5}^4]\\
&\lesssim n\E[\bar{X}^{16}]+n^2(\E[\bar{X}^{14}]\E[\bar{X}^2]+\E[\bar{X}^{12}]\E[\bar{X}^4]+\E[\bar{X}^{10}]\E[\bar{X}^6]+\E[\bar{X}^{8}]^2)\\
&+n^3(\E[\bar{X}^{12}]\E[\bar{X}^2]^2+\E[\bar{X}^{10}]\E[\bar{X}^4]\E[\bar{X}^2]+\E[\bar{X}^{8}]\E[\bar{X}^4]^2+\E[\bar{X}^{8}]\E[\bar{X}^6]\E[\bar{X}^2]+\E[\bar{X}^{6}]^2\E[\bar{X}^4])\\
&+n^4(\E[\bar{X}^{10}]\E[\bar{X}^2]^3+\E[\bar{X}^{8}]\E[\bar{X}^4]\E[\bar{X}^2]^2+\E[\bar{X}^{6}]^2\E[\bar{X}^2]^2+\E[\bar{X}^{6}]\E[\bar{X}^4]^2\E[\bar{X}^2]+\E[\bar{X}^{4}]^4)\\
&+n^5(\E[\bar{X}^{8}]\E[\bar{X}^2]^4+\E[\bar{X}^{6}]\E[\bar{X}^4]\E[\bar{X}^2]^3+\E[\bar{X}^{4}]^3\E[\bar{X}^2]^2)=o(n^4p^3)+o(n^6p)\,.
\end{align*} 
Consequently, we obtain
\begin{align*}
\frac{1}{n^8\sigma_n^4}\sum_{j=1}^p\E[M_j^4]=o\Big(\frac{p^6}{n^6\sigma_n^4}\Big)+o\Big(\frac{p^2}{n^2\sigma_n^4}\Big)=o(1),
\end{align*}
since $\sigma_n^4\asymp (p/n)^6+(p/n)^2$ if $\E[X^4]\neq 1$.  In the Bernoulli case it holds that $M_{j1}=M_{j3}=M_{j4}=M_{j5}=0$ and $\E[M_{j2}^4]=(j-1)^2 O(n^4)$. Therefore, $\frac{1}{n^8\sigma_n^4}\sum_{j=1}^p\E[M_j^4]=O(p^3/(n^4\sigma_n^4))=o(1)$ as $\sigma_n^4\asymp (p/n)^4$, which finishes the proof.\\

\subsection{Proof of Theorem \ref{thm:non-clt}}


We split the modified trace of $\S^2$ into four terms 
\begin{align*}
&\frac{n^2}{a_{np}^4}\tr(\S^2)-\frac{2n(n+p-2)}{a_{np}^4}\tr(\S)+\frac{np(n+p-2)}{a_{np}^4}\\
 &=\frac{1}{a_{np}^4}\sum_{i=1}^p \sum_{t=1}^n X_{it}^4+\frac{2}{a_{np}^4}\sum_{1\leq i<j\leq p}\sum_{t=1}^n(X_{it}^2-1)(X_{jt}^2-1)\\
 &\quad +\frac{2}{a_{np}^4}\sum_{i=1}^p\sum_{1\leq t<u\leq n}(X_{it}^2-1)(X_{iu}^2-1)+\frac{2}{a_{np}^4}\sum_{1\leq i<j\leq p}\sum_{\substack{t,u=1\\u\neq t}}^nX_{it}X_{jt}X_{iu}X_{ju} \\
&=:\frac{1}{a_{np}^4}\sum_{i=1}^p \sum_{t=1}^n X_{it}^4+V_{n,1}+V_{n,2}+V_{n,3}.
\end{align*}
As $X^4$ is regularly varying with index $\alpha/4$ and
\begin{align*}
    \frac{np}{a_{np}^4}\E[X^4\mathds{1}_{\{X^4\leq a_{np}^4\}}]\sim \frac{\alpha}{4-\alpha}\,np\,\P(X^4>a_{np}^4)\sim \frac{\alpha}{4-\alpha}\,, \qquad \nto\,,
\end{align*}
the first term converges to an $\alpha/4$-stable distribution by \cite[Theorem 3.8.2]{durrett2019probability} (see also \cite{samorodnitsky:taqqu:1994}). By \cite[p.~164]{durrett2019probability} this $\alpha$-stable distribution has the characteristic function 
\begin{align*}
    \E[e^{\mathrm{i}t\zeta_{\alpha/4}}]=\exp\Big(\mathrm{i}tc_\alpha+\frac{\alpha}{4}\int_0^\infty\big(e^{\mathrm{i}tx}-1-\frac{\mathrm{i}tx}{1+x^2}\big)x^{-(\alpha/4+1)}dx\Big),
\end{align*}
where $c_\alpha$ is a constant only depending on $\alpha$.

Therefore, it suffices to show that $V_{n,1}$, $V_{n,2}$ and $V_{n,3}$ tend to zero in probability. 
We will start by showing $V_{n,3}\cip 0$. As $X$ is regularly varying with index $\alpha>2$, the second moment of $X$ exists by Proposition 1.3.2 of \cite{mikosch1999regular}. An application of Markov's inequality yields for $\delta>0$
\begin{align*}
\P\big({|V_{n,3}|}>\delta\big)
&\leq \frac{4}{\delta^2a_{np}^8}\E\Big[\Big(\sum_{1\leq i<j\leq p}\sum_{\substack{t,u=1\\u\neq t}}^nX_{it}X_{jt}X_{iu}X_{ju}\Big)^2\Big]\\
&\lesssim { \frac{1}{a_{np}^8}\sum_{1\leq i<j\leq p}\sum_{\substack{t,u=1\\u\neq t}}^n\E[X^2]^4\lesssim \frac{p^2n^2}{a_{np}^8}=o(1)\,,}
\end{align*}
where we used that $a_{np}= (np)^{1/\alpha}\ell_1(np)$ for $\alpha\in (2,4)$ (see \cite{bingham:goldie:teugels:1987}) and a slowly varying function $\ell_1$. The last step is a consequence of the following property of slowly varying functions $\ell$. By the Potter Bounds, which can be found in Theorem 1.5.6 of \cite{bingham:goldie:teugels:1987}, it holds that for $x$ sufficiently large and any $\epsilon>0$ and $K>1$
\begin{align}\label{regprop1}
K x^{-\epsilon}\leq \ell(x)\leq K x^\epsilon.
\end{align}

To show that $V_{n,1}\cip 0$ we will truncate the random variables $X_{it}^2$ 
at $s_n:=p^{2/\alpha}n^{2(1+\epsilon)/\alpha}$ for a positive $\epsilon$ sufficiently small. 
Then we have for the event $Q:=\bigcup\limits_{i=1}^p\bigcup\limits_{t=1}^n\{X_{it}^2>s_n\}$ that
\begin{align*}
\P\big(Q\big)\leq pn \,\P(|X|>s_n^{1/2})\to 0\,,\qquad \nto\,,
\end{align*}
where we used that $|X|$ is regularly varying with index $\alpha\in(2,4)$. 
Letting $\bar{X}^2_{it}:=X_{it}^2\1_{\{X_{it}^2\leq s_n\}}$ and $\eta_{it}:=\bar{X}_{it}^2-\E[\bar{X}_{it}^2]$, we get for any $\delta>0$
 \begin{align}
&\P\Big(|V_{n,1}|>\delta\Big)
\leq \P\Big(|V_{n,1}|>\delta,\,Q^C\Big)+ \P(Q)\nonumber\\
&\leq \P\Big(\Big|\frac{2}{a_{np}^4}\sum_{1\leq i<j\leq p}\sum_{t=1}^n(\bar{X}_{it}^2-1)(\bar{X}_{jt}^2-1)\Big| >\delta\Big)+o(1)\\
&\leq \P\Big(\Big|\frac{2}{a_{np}^4}\sum_{1\leq i<j\leq p}\sum_{t=1}^n\eta_{it}\eta_{jt}\Big| >\frac{\delta}{2}\Big)+\P\Big(\Big|\frac{2(p-1)(\E[\bar{X}^2]-1)}{a_{np}^4}\sum_{i=1}^{p}\sum_{t=1}^n\eta_{it}+q_n\Big|>\frac{\delta}{2}\Big)+o(1)\\
&\lesssim \frac{n}{a_{np}^8}\Var\Big(\sum_{1\leq i<j\leq p}\eta_{i1}\eta_{j1}\Big)+\frac{p^3n(\E[\bar{X}^2]-1)^2}{a_{np}^8}\Var(\eta_{11}) +o(1)\,,\label{repl}
\end{align}
where for the last line we used Markov's inequality and the fact that 
\begin{align*}
q_n:=\frac{p(p-1) n}{a_{np}^4} (1-\E[\bar{X}^2])^2 \leq \frac{p^2 n}{a_{np}^4} \E[X^2\1_{\{X^2>s_n\}}]^2\to 0\,,\qquad \nto\,,
\end{align*}
 since $\E[X^2\mathds{1}_{\{X^2>s_n\}}]\asymp s_n\P(X^2>s_n)$ by Karamata's Theorem (see \cite{bingham:goldie:teugels:1987}) and due to the Potter bounds \eqref{regprop1}.
 
 Regarding the first term in \eqref{repl}, we obtain that
 \begin{align*}
  \frac{n}{a_{np}^8}\Var\Big(\sum_{1\leq i<j\leq p}\eta_{i1}\eta_{j1}\Big) 
  &\lesssim \frac{np^2}{a_{np}^8} \E[\eta_{11}^2]^2\sim \frac{np^2}{a_{np}^8} \E[\bar{X}^4]^2\to 0\,,\qquad \nto\,,
 \end{align*}
 since $\E[X^4\mathds{1}_{\{X^2\leq s_n\}}]\asymp s_n^{2}\P(X^2>s_n)$ by Karamata's theorem. Using Karamata's theorem again and similiar arguments as above, also the second term of \eqref{repl} tends to zero as $n\to\infty$.
 Therefore, $V_{n,1}\cip 0$ as $n\to\infty$.
 
Notice that $V_{n,2}$ is equal to $V_{n,1}$ if we exchange the roles of $n$ and $p$, which does not matter for the proof given above. Hence, it also holds that $V_{n,2}\cip 0$ as $n\to\infty$.

\subsection{Proof of Theorem \ref{thm:kal}}

Let now $y\in\R$ and $l_1,\ldots,l_k\in\N_0$. Let $(Q_n)_n$ be a sequence of probability measures defined by the distribution functions 
\begin{align*}
\P(Y_n\leq y,\,N_n(B_1)\leq l_1,\ldots,\,N_n(B_k)\leq l_k),
\end{align*}
where $B_1,\ldots, B_k\in\mathcal{B}_N$. We recall that for a point process $\xi$, $\mathcal{B}_{\xi}:=\{B\; \text{bounded Borel set}: \xi(\partial B)=0\}$. As $\R^{k+1}$ is a Polish space 
$(Q_n)_n$ is tight. Then, by Prokhorov's Theorem, $(Q_n)_n$ is relatively compact with respect to convergence in distribution and hence, for every sequence $(n_m)_{m\in \N}$ in $\N$ there exists a subsequence $(n_{m_j})_{j\in\N}$ with
\begin{align}\label{conv}
&\lim\limits_{j\to\infty}\P(Y_{n_{m_j}}\leq y, N_{n_{m_j}}(B_1)\leq l_1,\ldots,\,N_{n_{m_j}}(B_k)\leq l_k)\\
&=\P(\tilde{Y}\leq y,\tilde{N}(B_1)\leq l_1,\ldots,\, \tilde{N}(B_k)\leq l_k),
\end{align}
where $\tilde{Y}, \tilde{N}(B_1),\ldots,\tilde{N}(B_k)$ are real valued random variables. Since the sets $B_1,\ldots, B_k$ are arbitrary, we can also find a subsequence $(n_{m_j})_{j\in\N}$ such that \eqref{conv}
holds for any choice of $B_1,\ldots, B_k\in\mathcal{B}_N$. Let $\tilde{N}$ be the point process defined by the random vectors $(\tilde{N}(B_1),\ldots,\tilde{N}(B_k))$ for $B_1,\ldots, B_k\in \mathcal{B}_N$.

Assumption (K2) implies for every $U\in\mathcal{U}$
\begin{align*}
\liminf\limits_{n\to\infty}\P(N_n(U)=0)\geq\P(N(U)=0). 
\end{align*}
Therefore, we get by Lemma 4.6 of \cite{kallenberg:1983} that $\mathcal{B}_N\subset\mathcal{B}_{\tilde{N}}$ and hence $\mathcal{U}\subset\mathcal{B}_{\tilde{N}}$.
Then, we get
\begin{align*}
\P(Y\leq y,\, N(U)=0)=\lim\limits_{j\to\infty}\P(Y_{n_{m_j}}\leq y,\, N_{n_{m_j}}(U)=0)=\P(\tilde{Y}\leq y,\, \tilde{N}(U)=0)
\end{align*}
for every $U\in \mathcal{U}$.
Let $\mathfrak{R}$ be the set of locally finite measures $\mu$  on $(\R,\mathcal{B})$, where $\mathcal{B}$ consists of all bounded Borel sets  with $\mu(B)\in \N_0$ for all $B\in\mathcal{B}$. Additionally, $\mathcal{N}$ is the $\sigma$-algebra on $\mathfrak{R}$ that is generated by the mappings $\mu\mapsto\mu(B),\,B\in\mathcal{B}$, i.e., the smallest $\sigma$-algebra making these mappings measurable.

We now introduce the Dynkin-system
\begin{align*}
\mathcal{D}:=\{M\in\mathcal{N}:\, \P(Y\leq y, N\in M)=\P(\tilde{Y}\leq y,\,\tilde{N}\in M)\}.
\end{align*}
As $Y_n\cid Y$, we have $\P(Y\leq y)=\P(\tilde{Y}\leq y)$ and therefore, $\mathfrak{R}\in \mathcal{D}$. Moreover, $\mathcal{D}$ is closed under proper differences and monotone limits. Let 
\begin{align*}
\mathcal{C}:=\{\{\mu\in\mathfrak{R}:\,\mu(U)=0\},\, U\in\mathcal{U}\}.
\end{align*}
By assumption (K2), $\mathcal{C}\subset\mathcal{D}$ and $\mathcal{C}$ is closed under finite intersection. Therefore, by 15.2.1 of \cite{kallenberg:1983} it follows that $\sigma(\mathcal{C})\subset\mathcal{D}$. By Lemmas 2.2, 1.3 and 1.4 of \cite{kallenberg:1983} it holds that $\varphi:\mu\to\mu^*$ is measurable $\sigma(\mathcal{C})\to\mathcal{N}$, where $\mu^*(B)=\sum_{s\in B}\mathds{1}_{[1,\infty)}(\mu\{s\})$ for every $B\in\mathcal{B}$. As $\sigma(\mathcal{C})\subset\mathcal{D}$ we get for every $M\in\mathcal{N}$
\begin{align*}
\P(Y\leq y, N^*\in M)&=\P(Y\leq y, N\in\varphi^{-1}(M))\\
&=\P(\tilde{Y}\leq y, \tilde{N}\in\varphi^{-1}(M))\\
&=\P(\tilde{Y}\leq y, \tilde{N}^*\in M)\,.
\end{align*}
A simple point process $\mu$ can be written as
\begin{align*}
\mu=\sum_{i\in I}\varepsilon_{X_i},
\end{align*}
where $I$ is an index set and the $X_i$'s are random elements. 
Therefore, it holds for every $B\in\mathcal{B}$ that
\begin{align*}
\mu^*(B)=\sum_{s\in B}\mathds{1}_{[1,\infty)}(\mu\{s\})=\sum_{s\in B}\mu\{s\}=\sum_{i\in I}\sum_{s\in B}\mathds{1}_{\{X_i=s\}}=\sum_{i\in I}\mathds{1}_{\{X_i\in B\}}=\mu(B).
\end{align*}
As $N$ is simple, we get
\begin{align*}
\P(Y\leq y,\, N\in M)=\P(\tilde{Y}\leq y,\, \tilde{N}^*\in M),
\end{align*}
for every $M\in\mathcal{N}$ and every $y\in\R$. Therefore, we also get $\P(N\in M)=\P(\tilde{N}^*\in M)$ for every $M\in\mathcal{N}$. We define the set of measures
\begin{align*}
\hat{M}:=\{\mu\in\mathfrak{R}:\, \mu(B_1)\leq l_1,\ldots, \mu(B_k)\leq l_k\}\in\mathcal{N}.
\end{align*} 
Then, it follows that
\begin{align*}
\P(N(B_1)\leq l_1,\ldots, N(B_k)\leq l_k)=\P(\tilde{N}^*(B_1)\leq l_1,\ldots, \tilde{N}^*(B_k)\leq l_k)
\end{align*}
and as $B_1,\ldots, B_k,l_1,\ldots, l_k$ were chosen arbitrarily, we have $N\stackrel{d}{=}\tilde{N}^*$. Additionally, for $I\in\mathcal{J}$ it holds that as $j\to\infty$
\begin{align}\label{h1}
\P(N_{n_{m_j}}(I)\leq l)\to\P(\tilde{N}(I)\leq l).
\end{align}
Then, by $N\stackrel{d}{=}\tilde{N}^*$, the definition of $\tilde{N}^*$, \eqref{h1}, 15.4.3 of \cite{kallenberg:1983} and assumption (K1) we get
\begin{align*}
\E[N(I)]=\E[\tilde{N}^*(I)]\leq \E[\tilde{N}(I)]\leq \liminf\limits_{j\to\infty}\E[N_{n_{m_j}}(I)]\leq \limsup\limits_{n\to\infty}\E[N_n(I)]\leq \E[N(I)].
\end{align*}
Therefore, $\tilde{N}$ is a.s. simple and consequently
\begin{align*}
\P(Y\leq y,\, N\in M)=\P(\tilde{Y}\leq y,\, \tilde{N}\in M)
\end{align*}
for every $M\in\mathcal{N}$. Inserting $\hat{M}$, we get
\begin{align*}
\P(Y\leq y, N(B_1)\leq l_1,\ldots, N(B_k)\leq l_k)=\P(\tilde{Y}\leq y,\tilde{N}(B_1)\leq l_1,\ldots, \tilde{N}(B_k)\leq l_k).
\end{align*}
As the subsequence $(n_m)_m$ was arbitrary, we conclude for every $l_1,\ldots,l_k\in\N_0$ and $B_1,\ldots, B_k\in\mathcal{B}_N$
\begin{align*}
\lim\limits_{n\to\infty}\P(Y_{n}\leq y,\, N_{n}(B_1)\leq l_1,\ldots,\,N_{n}(B_k)\leq l_k)&=\P({Y}\leq y,\,{N}(B_1)\leq l_1,\ldots,\, {N}(B_k)\leq l_k)\\
&=\P({Y}\leq y)\P({N}(B_1)\leq l_1,\ldots,\, {N}(B_k)\leq l_k)
\end{align*}
and therefore $(Y_n)_n$ and $(N_n)_n$ are asymptotically independent.

\subsection{Proof of Theorem \ref{thm:ind}}\label{sec:3.4}

By the proof of Theorem~\ref{lem:pp} we know that $\lim\limits_{n\to\infty}\E[N_n(I)]=\E[N(I)]$ for every interval $I$. In view of Theorem \ref{thm:kal}, it suffices to show
\begin{align*}
\lim\limits_{n\to\infty} \P(Z_n\leq y,\,N_n(U)=0)=\Phi(y)\P(N(U)=0),
\end{align*}
for every $y\in\R$ and $U\in\mathcal{U}$, where $\Phi$ is the distribution function of the standard normal distribution and $\mathcal{U}$ is the set of finite unions of intervals. As $Z_n\cid \mathcal{N}(0,1)$, this is equivalent to 
\begin{align}\label{eq:tosh}
\lim\limits_{n\to\infty} \P(Z_n\leq y,\,N_n(U)\neq 0)=\Phi(y)\P(N(U)\neq 0),
\end{align}
for every $U\in\mathcal{U}$ and $y\in \R$. Throughout this proof, we let $U\in\mathcal{U}$ and $y\in \R$ be arbitrary.
\medskip

Recall that $p=O\big(n^{(s-2)/4}\big)$ for some $s\ge 4$ and that $\vep>0$ is such that $\E|X|^{s+\vep}]<\infty$. We need the following notation. For $\tilde{s}=s+\vep$ let $(\beta_n)_n$ be a positive sequence, which tends to zero and satisfies $\beta_n \gg (\E[|X|^{\tilde{s}}\mathds{1}_{\{|X|>\beta_n(np)^{1/\tilde{s}}\}}])^{1/\tilde{s}}$. Such a sequence exists by similar reasons as in the proof of Theorem \ref{thm:clt}. We set
\begin{align*}
\bar{X}_{it}&:= X_{it}\mathds{1}_{\{|X_{it}|\leq \beta_n (np)^{1/\tilde{s}}\}},\\
\xi_{ijt}&:=X_{it}X_{jt},\quad \bar{\xi}_{ijt}:=\bar{X}_{it}\bar{X}_{jt},\quad 1\leq i,\,j\leq p,\,1\leq t\leq n,\\
T_{ij}&:=\sum_{t=1}^n\xi_{ijt},\quad \bar{T}_{ij}:=\sum_{t=1}^n\bar{\xi}_{ijt},\quad 1\leq i,\,j\leq p.
\end{align*}
By Lemma \ref{lem:hilf}, we have for $\delta>0$
\begin{align}
&\P(\bar{Z}_n\leq y-\delta,\,N_n(U)\neq 0)-\P(|Z_n-\bar{Z}_n|>\delta)
\leq\P(Z_n\leq y,\,N_n(U)\neq 0) \nonumber\\
&\leq \P(\bar{Z}_n\leq y+\delta,\,N_n(U)\neq 0)+\P(|Z_n-\bar{Z}_n|>\delta)\,,\label{eq:dfsss}
\end{align}
where 
\begin{align}\label{eq:defZbar}
\bar{Z}_n &=\frac{1}{n^2\sigma_n}\sum_{i,j=1}^p(\bar{T}_{ij}^2-\E[\bar{T}_{ij}^2])\,.
\end{align}
Lemma \ref{lem:cut} asserts that $\lim_{\nto}\P(|Z_n-\bar{Z}_n|>\delta)= 0$ for every $\delta>0$.

Assume for the moment that 
\begin{align}\label{eq:dddg}
\lim\limits_{n\to\infty} \P(\bar{Z}_n\leq y,\,N_n(U)\neq 0)=\Phi(y)\P(N(U)\neq 0).
\end{align}
In conjunction with \eqref{eq:dfsss}, this yields for $\delta>0$
\begin{align*}
\limsup\limits_{n\to\infty} \P(Z_n\leq y,\,N_n(U)\neq 0)&\leq \Phi(y+\delta)\P(N(U)\neq 0),\\
\liminf\limits_{n\to\infty} \P(Z_n\leq y,\,N_n(U)\neq 0)&\geq \Phi(y-\delta)\P(N(U)\neq 0),
\end{align*}
so that taking the limit $\delta\to 0$ establishes \eqref{eq:tosh} by the continuity of the normal distribution. 
Therefore, it remains to show \eqref{eq:dddg} for which we proceed similarly to \cite{feng2022max}.
\medskip

To this end, we set $A_n=A_n(y)=\{\bar{Z}_n\leq y\}$ and $B_I=B_I(U)=\{d_p(\sqrt{n}S_{ij}-d_p)\in U\}$, where $I=(i,j)\in\Lambda_n:=\{(i,j):\,1\leq i<j\leq p\}$. For $I_1=(i_1,j_1)\in\Lambda_n$ and $I_2=(i_2,j_2)\in\Lambda_n$ we write $I_1<I_2$ if $i_1<i_2$ or ($i_1=i_2$ and $j_1<j_2$). Then we have
\begin{align*}
\P(\bar{Z}_n\leq y,\,N_n(U)\neq 0)&=\P\Big(\bigcup_{I\in\Lambda_n}A_nB_I\Big),\\
\P(\bar{Z}_n\leq y) \P(N_n(U)\neq 0)&=\P(A_n) \P\Big(\bigcup_{I\in\Lambda_n}B_I\Big),
\end{align*}
where $A_nB_I=A_n\cap B_I$. 
Using $Z_n-\bar{Z}_n\cip 0$, Theorem~\ref{thm:clt} and Theorem~\ref{lem:pp}, it holds 
\begin{align}
\lim_{n\to\infty} \P(\bar{Z}_n\leq y)=\Phi(y) \quad \text{ and }\quad \lim_{\nto} \P(N_n(U)\neq 0)=\P(N(U)\neq 0)\,,
\end{align}
so that equation \eqref{eq:dddg} follows from
\begin{equation}\label{eq:jh2905}
\lim_{\nto} \P\Big(\bigcup_{I\in\Lambda_n}A_nB_I\Big)-\P(A_n) \P\Big(\bigcup_{I\in\Lambda_n}B_I\Big) =0\,.
\end{equation}
To prove \eqref{eq:jh2905}, we start by setting $$W_{n,d}:=\sum_{I_1<\ldots<I_d}\P\Big(\bigcap\limits_{l=1}^d A_nB_{I_l}\Big)\,, \quad \wt W_{n,d}:=\sum_{I_1<\ldots<I_d}\P\Big(\bigcap\limits_{l=1}^d B_{I_l}\Big) \quad \text{ and } \quad  \bar{W}_{n,d}:=\P(A_n)\wt W_{n,d}\,.$$
Then the Bonferroni bounds yield for $k\geq 1$
\begin{equation}\label{eq:bonferroni}
\begin{split}
\sum_{d=1}^{2k}(-1)^{d-1}W_{n,d}\leq\P\Big(\bigcup\limits_{I\in\Lambda_n}A_nB_I\Big)\leq \sum_{d=1}^{2k-1}(-1)^{d-1}W_{n,d}\,,\\
\sum_{d=1}^{2k}(-1)^{d-1}\bar{W}_{n,d}\leq \P(A_n)\P\Big(\bigcup\limits_{I\in\Lambda_n}B_I\Big)\leq\sum_{d=1}^{2k-1}(-1)^{d-1}\bar{W}_{n,d}\,.
\end{split}
\end{equation}
By \eqref{eq:bonferroni}, we have for $k\geq 1$
\begin{align}
\left| \P\Big(\bigcup_{I\in\Lambda_n}A_nB_I\Big)-\P(A_n)\P\Big(\bigcup\limits_{I\in\Lambda_n}B_I\Big)\right| &\le \sum_{d=1}^{2k-1} \big|W_{n,d}-\bar{W}_{n,d}\big| +\max\{W_{n,2k},\bar{W}_{n,2k}\} \nonumber\\
&\le \sum_{d=1}^{2k-1} \big|W_{n,d}-\bar{W}_{n,d}\big| +\wt W_{n,2k}\,.\label{eq:sdsdgdf}
\end{align}
From \cite[p.~555]{heiny:mikosch:yslas:2021} we know that 
\begin{equation}\label{eq:gsedgds}
\lim_{n\to\infty}\wt W_{n,k}=\frac{(\mu(U))^k}{k!}\,, \qquad k\ge 1.
\end{equation}
Note that $\lim_{k\to\infty} \frac{(\mu(U))^k}{k!} =0$. By first letting $\nto$ and then $k\to \infty$ in \eqref{eq:sdsdgdf} we now obtain \eqref{eq:jh2905} provided that
\begin{equation}\label{eq:finaljh}
\lim_{n\to\infty}\big|W_{n,d}-\bar{W}_{n,d}\big|=0\,,\qquad d\ge 1\,.
\end{equation}

\textbf{Proof of \eqref{eq:finaljh}.} 
For fixed $I_1<\ldots<I_d\in\Lambda_n$ with $I_l=(i_l,j_l)$ for $l=1,\ldots, d$ we will identify the summands of $\bar{Z}_n$ in \eqref{eq:defZbar} that are dependent on $B_{I_1},\ldots,B_{I_d}$ and show that their contribution  is negligible (in a suitable way). Therefore, we introduce the set $\Lambda_{n,d}=\Lambda_{n,d}(I_1,\ldots, I_d)$ through
\begin{align*}
\Lambda_{n,d}:=\{(i_l,j),\, (j,i_l):\,1\leq j\leq p,\, 1\leq l\leq d\}\cup \{(i,j_l),\, (j_l,i):\, 1\leq i\leq p,\, 1\leq l\leq d\}.
\end{align*}
The set $\Lambda_{n,d}$ includes the indices $(i,j)$ of all summands of $\bar{Z}_n$ that are dependent on $B_{I_1},\ldots,B_{I_d}$. Notice that $\Lambda_{n,d}$ is not a subset of $\Lambda_n$ because $\Lambda_{n,d}$ might also contain indices $(i,i)$ corresponding to diagonal elements $S_{ii}$ of the covariance matrix. 
For our further arguments the following bound on the cardinality of $\Lambda_{n,d}$ is important:  $$|\Lambda_{n,d}|\leq 4dp.$$
By the definition of $\Lambda_{n,d}$, $\bar{Z}_n-\bar{Z}_{n,d}$ is independent of $B_{I_1},\ldots, B_{I_d}$, where
\begin{align*}
\bar{Z}_{n,d}:=\frac{1}{n^2\sigma_n}\sum_{(i,j)\in\Lambda_{n,d}}\big(\bar{T}_{ij}^2-\E[\bar{T}_{ij}^2]\big)\,. 
\end{align*}
 Using this independence and applying Lemma \ref{lem:hilf} twice we get for $\delta>0$,
\begin{align*}
\P\Big(A_n(y)\bigcap\limits_{l=1}^d B_{I_l}\Big)&\leq\P\Big(\bar{Z}_n-\bar{Z}_{n,d}\leq y+\delta,\,\bigcap\limits_{l=1}^d B_{I_l}\Big)+\P(|\bar{Z}_{n,d}|>\delta)\\
&=\P(\bar{Z}_n-\bar{Z}_{n,d}\leq y+\delta)\,\P\Big(\bigcap\limits_{l=1}^d B_{I_l}\Big)+\P(|\bar{Z}_{n,d}|>\delta)\\
&\leq\P(A_n(y+2\delta))\,\P\Big(\bigcap\limits_{l=1}^d B_{I_l}\Big)+2\,\P(|\bar{Z}_{n,d}|>\delta)
\end{align*}
and similarly,
\begin{align*}
\P\Big(A_n(y)\bigcap\limits_{l=1}^d B_{I_l}\Big) \geq\P(A_n(y-2\delta))\,\P\Big(\bigcap\limits_{l=1}^d B_{I_l}\Big)-2\,\P(|\bar{Z}_{n,d}|>\delta).
\end{align*}
Therefore, it follows that
\begin{align*}
&\Big|\P\Big(A_n(y)\bigcap\limits_{l=1}^d B_{I_l}\Big)-\P(A_n(y))\P\Big(\bigcap\limits_{l=1}^d B_{I_l}\Big)\Big|\\
&\leq 2 \Big(\P(A_n(y+2\delta))-\P(A_n(y-2\delta))\Big)\,\P\Big(\bigcap\limits_{l=1}^d B_{I_l}\Big)+4\,\P(|\bar{Z}_{n,d}|>\delta).
\end{align*}
If we assume 
\begin{align}
    \sum_{I_1<\ldots<I_d}\P(|\bar{Z}_{n,d}|>\delta)\to 0,
    \label{wk0}
\end{align}
we get using Theorem \ref{thm:clt} and \eqref{eq:gsedgds} that
\begin{align*}
\limsup\limits_{n\to\infty}|W_{n,d}-\bar{W}_{n,d}|\leq 2(\Phi(y+2\delta)-\Phi(y-2\delta))\, \frac{(\mu(U))^d}{d!}.
\end{align*}
Sending $\delta$ to zero establishes \eqref{eq:finaljh}. Therefore it remains to show \eqref{wk0}.
\medskip

By Markov's inequality we obtain for even $\tau\in \N$ and $\delta >0$,
{\small\begin{align}
\P(|\bar{Z}_{n,d}|>{2}\delta)
&\le \frac{1}{n^{2\tau}\sigma_n^\tau\delta^\tau}\bigg(\!\!\E\Big[\Big(\!\!\sum_{\substack{(i,j)\in\Lambda_{n,d}\\ i\neq j}}\!\!\!\!(\bar{T}_{ij}^2-\E[\bar{T}_{ij}^2])\Big)^\tau\Big]+\E\Big[\Big(\!\!\sum_{(i,i)\in\Lambda_{n,d}}\!\!\!\!(\bar{T}_{ii}^2-\E[\bar{T}_{ii}^2])\Big)^\tau\Big]\!\!\bigg).\label{splt}
\end{align}}


Letting $\mathcal{K}:=\{i_l, j_l\,|\,l=1,\ldots,d\}$ we write the first term on the \rhs\ of \eqref{splt} as follows
{\small\begin{align}
&\frac{1}{n^{2\tau}\sigma_n^\tau\delta^\tau}\E\Big[\Big(\sum_{\substack{(i,j)\in\Lambda_{n,d}\\ i\neq j}}(\bar{T}_{ij}^2-\E[\bar{T}_{ij}^2])\Big)^\tau\Big]
=\frac{2^\tau}{n^{2\tau}\sigma_n^\tau\delta^\tau}\E\Big[\Big(\sum_{k\in\mathcal{K}}\sum_{\substack{i=1\\i\neq k}}^{p}(\bar{T}_{ik}^2-\E[\bar{T}_{ik}^2])\Big)^\tau\Big]\nonumber\\
&\le \frac{2 (8d)^{\tau-1}}{n^{2\tau}\sigma_n^\tau\delta^\tau}\sum_{k\in\mathcal{K}}\bigg(\E\Big[\Big(\sum_{\substack{i=1\\i\neq k}}^{p}(\bar{T}_{ik}^2-\E[\bar{T}_{ik}^2|\tx_{k}])\Big)^\tau\Big]+\E\Big[\Big(\sum_{\substack{i=1\\i\neq k}}^{p}(\E[\bar{T}_{ik}^2|\tx_{k}]-\E[\bar{T}_{ik}^2])\Big)^\tau\Big]\bigg), \label{cond}
\end{align}}
where we used the Marcinkiewicz-Zygmund inequality (see \cite[Theorem~2, p.~386]{teicher1978probability}) and the inequality $(a+b)^c\leq 2^{c-1}(a^c+b^c)$ for $a>0$, $b>0$ and $c\geq 1$.
We apply the law of total expectation and the Marcinkiewicz-Zygmund inequality to the first term of \eqref{cond} and obtain
\begin{align*}
&\E\Big[\E\Big[\Big(\sum_{\substack{i=1\\i\neq k}}^{p}(\bar{T}_{ik}^2-\E[\bar{T}_{ik}^2|\tx_{k}])\Big)^\tau\Big|\tx_{k}\Big]\Big]\nonumber\leq K_\tau (p-1)^{\tau/2-1}\sum_{\substack{i=1\\i\neq k}}^{p}\E[(\bar{T}_{ik}^2-\E[\bar{T}_{ik}^2|\tx_{k}])^\tau]
\end{align*}
where $K_\tau$ is a constant only depending on $\tau$, which may vary from line to line.
Recalling the notation $\bar{T}_{ij}=\sum_{t=1}^n \bar{\xi}_{ijt}$ we may write
\begin{align}
&\E[(\bar{T}_{ik}^2-\E[\bar{T}_{ik}^2|\tx_{k}])^\tau]=\nonumber\\
&\sum_{t_1,u_1,\ldots,t_\tau,u_\tau=1}^n \!\!\!\!\!\!\!\E[(\bar{\xi}_{ikt_1}\bar{\xi}_{iku_1}-\E[\bar{\xi}_{ikt_1}\bar{\xi}_{iku_1}|\tx_{k}])\ldots(\bar{\xi}_{ikt_\tau}\bar{\xi}_{iku_\tau}-\E[\bar{\xi}_{ikt_\tau}\bar{\xi}_{iku_\tau}|\tx_{k}])].\label{fac}
\end{align}
If more than $\tau+1$ of the indices $t_1,u_1,\ldots,\,t_\tau,u_\tau$ are different, then there exists a tuple $(t_k,u_k)$ such that $t_k\neq t_l$, $t_k\neq u_l$, $u_k\neq t_l$ and $u_k\neq u_l$ for every $l\neq k$ and therefore one of the factors in the mean of \eqref{fac} is independent, so that the summand disappears.
The remaining summands are bounded above by 
\begin{align}\label{mea}
C|\E[\bar{\xi}_{ikt_1}\bar{\xi}_{iku_1}\cdots\bar{\xi}_{ikt_\tau}\bar{\xi}_{iku_\tau}]|=C\E[\bar{X}_{it_1}\bar{X}_{iu_1}\cdots\bar{X}_{it_\tau}\bar{X}_{iu_\tau}]^2.
\end{align}
Similar to \eqref{mo1}, \eqref{mo2} and \eqref{mo3} it holds that $\E[\bar{X}]=o((np)^{-(\tilde{s}-1)/\tilde{s}})$, $\E[\bar{X}^r]\leq C$ for $r\leq \tilde{s}$ and $\E[\bar{X}^r]= o((np)^{(r-\tilde{s})/\tilde{s}})$ for $r>\tilde{s}$. 
From this we deduce that if $|\{t_1,u_1,\ldots,t_{\tau}, u_{\tau}\}|=\ell$ one has
\begin{align}
\E[\bar{X}_{it_1}\bar{X}_{iu_1}\cdots\bar{X}_{it_\tau}\bar{X}_{iu_\tau}]^2\lesssim  (\E[\bar{X}^2]^{\ell-1}\E[\bar{X}^{2\tau-2\ell+2}])^2=o((np)^{2(2\tau-2\ell+2-\tilde{s})/\tilde{s}}).\label{summand}
\end{align}
Therefore, we get for the first term of \eqref{cond} 
{\small\begin{align} 
\frac{2 (8d)^{\tau-1}}{n^{2\tau}\sigma_n^\tau\delta^\tau}\sum_{k\in\mathcal{K}}\E\Big[\Big(\sum_{\substack{i=1\\i\neq k}}^{p}(\bar{T}_{ik}^2-\E[\bar{T}_{ik}^2|\tx_{k}])\Big)^\tau\Big]&=\frac{p^{\tau/2}}{n^{2\tau}\sigma_n^\tau}\max_{1\leq \ell\leq \tau+1}o(n^\ell (np)^{(4\tau-4\ell+4-2\tilde{s})/\tilde{s}}).\label{ok2}
\end{align}}
The term in the maximum is either increasing or decreasing with $\ell$ or it is equal to $o((np)^{4\tau+4-2\tilde{s})/\tilde{s}})$ if $n\asymp (np)^{4/\tilde{s}}$.
Therefore, the \rhs\ in \eqref{ok2} is 
{\small\begin{align}
&\frac{p^{\tau/2}}{n^{2\tau}\sigma_n^\tau}\max\{o(n^{4\tau/\tilde{s}-1}p^{4\tau/\tilde{s}-2}),o(n^{\tau-1} p^{-2})\}=o\Big(\frac{n^{(s+2)\tau/\tilde{s}-1}p^{\tau/2-2}}{n^{2\tau}\sigma_n^\tau}\Big)=o(n^{(4-\tilde{s})\tau/(2\tilde{s})-1}),\phantom{abs}\label{summary1}
\end{align}}
where we used that $p=O(n^{(s-2)/4})$ and $(s+2)/\tilde{s}>1$ for an $\varepsilon<1$, and additionally,
 $\sigma_n^2\asymp (p/n)^3+ p/n$ if $\E[X^4]\neq 1$.\\  In the Bernoulli case $\E[X^4]= 1$, \eqref{mea} is bounded by a constant and therefore the first term of \eqref{cond} is $O(n/p^{\tau/2})$ as $\sigma_n^2\asymp (p/n)^2$.
 \medskip

%
%

By Jensen's inequality the second term of \eqref{cond} is bounded above by
\begin{align*}
&K'_\tau(p-1)^{\tau-1}\sum_{\substack{i=1\\i\neq k}}^{p}\E[(\E[\bar{T}_{ik}^2|\tx_{k}]-\E[\bar{T}_{ik}^2])^\tau],
\end{align*}
where $K'_\tau$ is a constant only depending on $\tau$, which may vary from line to line.
For the mean in the sum we can write
\begin{align*}
\E[(\E[\bar{T}_{ik}^2|\tx_{k}]-\E[\bar{T}_{ik}^2])^\tau]=\sum_{t_1,\, u_1,\ldots,t_\tau,u_\tau=1}^n &\E[((\bar{X}_{kt_1}\bar{X}_{ku_1}-\E[\bar{X}_{kt_1}\bar{X}_{ku_1}])\E[\bar{X}_{it_1}\bar{X}_{iu_1}])\cdots\\
&\times ((\bar{X}_{kt_\tau}\bar{X}_{ku_\tau}-\E[\bar{X}_{kt_\tau}\bar{X}_{ku_\tau}])\E[\bar{X}_{it_\tau}\bar{X}_{iu_\tau}])].
\end{align*}

We consider one of the summands above. There can be $0\leq q\leq \tau$ pairs of indices $(t_i,u_i)$ with $t_i=u_i$. Assume these pairs are $(t_1,u_1),\ldots, (t_q,u_q)$ and for $i>q$ it holds that $t_i\neq u_i$. Then the mean in the sum equals
\begin{align*}
&\E[\bar{X}]^{2(\tau-q)}\E[\bar{X}^2]^q\\
&\times\E[(\bar{X}_{kt_1}^2-\E[\bar{X}^2])\cdots(\bar{X}_{kt_q}^2-\E[\bar{X}^2])(\bar{X}_{kt_{q+1}}\bar{X}_{ku_{q+1}}-\E[\bar{X}]^2)\cdots(\bar{X}_{kt_\tau}\bar{X}_{ku_\tau}-\E[\bar{X}]^2)].
\end{align*}
If there are more than $\tau-q/2+1$ different indices, the summand is equal to zero. For $1\leq \ell\leq \tau-q/2+1$ different indices the summand is bounded above by 
\begin{align*}
C\E[\bar{X}]^{2(\tau-q)}\E[\bar{X}_{it_1}^2\cdots\bar{X}_{it_q}^2\bar{X}_{it_{q+1}}\bar{X}_{iu_{q+1}}\cdots\bar{X}_{it_\tau}\bar{X}_{iu_\tau}]&\lesssim  \E[\bar{X}]^{2(\tau-q)}(\E[\bar{X}^2]^{\ell-1}\E[\bar{X}^{2\tau-2\ell+2}])\\
&=o((np)^{(2\tau-2\ell+2-2(\tilde{s}-1)(\tau-q))/\tilde{s}}).
\end{align*}
Therefore, we get for the second term of \eqref{cond}
\begin{align}
&\frac{2 (8d)^{\tau-1}}{n^{2\tau}\sigma_n^\tau\delta^\tau}\sum_{k\in\mathcal{K}}\E\Big[\Big(\sum_{\substack{i=1\\i\neq k}}^{p}(\E[\bar{T}_{ik}^2|\tx_{k}]-\E[\bar{T}_{ik}^2])\Big)^\tau\Big]\nonumber\\
&=\frac{p^{\tau}}{n^{2\tau}\sigma_n^\tau}\max_{\substack{1\leq \ell\leq \tau-q/2+1\\ 0\leq q\leq \tau}}o(n^\ell (np)^{-2\ell/\tilde{s}}(np)^{(2\tau+2-2(\tilde{s}-1)(\tau-q))/\tilde{s}}).\label{oq}
\end{align}
The term in the maximum is either increasing or decreasing with $\ell$ or it is equal to\\ $o((np)^{2\tau+2-2(\tilde{s}-1)(\tau-q))/\tilde{s}})$ if $n\asymp (np)^{2/\tilde{s}}$.
Therefore, the expression in \eqref{oq} is
\begin{align*}
&\frac{p^{\tau}}{n^{2\tau}\sigma_n^\tau}\max\{\max_{0\leq q\leq \tau}o(n(np)^{(2(2-\tilde{s})\tau+2(\tilde{s}-1)q)/\tilde{s}}), \max_{0\leq q\leq \tau}o(n^{\tau-q/2+1}(np)^{((2\tilde{s}-1)q-2(\tilde{s}-1)\tau)/\tilde{s}})\}\\
&=\frac{p^{\tau}}{n^{2\tau}\sigma_n^\tau}\max\{o(n(np)^{2\tau/\tilde{s}}),o(n^{\tau/2+1}(np)^{\tau/\tilde{s}})\}=o\Big(\frac{n^{\tau/2+1}(np)^{\tau/\tilde{s}}p^{\tau}}{n^{2\tau}\sigma_n^\tau}\Big),
\end{align*}
since in the first step both terms in the maximum are growing with $q$, so that the terms are the largest for $q=\tau$, and since in the second step the last term of the maximum is larger than the first term for every $p=O(n^{(s-2)/4})$. As $\sigma_n^2\asymp p/n+(p/n)^3$ if $\E[X^4]\neq 1$, we have
\begin{align}
o\Big(\frac{n^{\tau/2+1}(np)^{\tau/\tilde{s}}p^{\tau}}{n^{2\tau}\sigma_n^\tau}\Big)
=o(n^{-(\tilde{s}-4)\tau/(2\tilde{s})+1})\label{summary2}.
\end{align} 
In the Bernoulli case the second term of \eqref{cond} is equal to zero. 
\medskip

For the second term of \eqref{splt} it holds that
\begin{align*}
&\frac{1}{n^{2\tau}\sigma_n^\tau\delta^\tau}\E\Big[\Big(\sum_{(i,i)\in\Lambda_{n,d}}(\bar{T}_{ii}^2-\E[\bar{T}_{ii}^2])\Big)^\tau\Big]\leq\frac{1}{n^{2\tau}\sigma_n^\tau\delta^\tau} (2d)^{\tau-1}\sum_{(i,i)\in\Lambda_{n,d}}\E[(\bar{T}_{ii}^2-\E[\bar{T}_{ii}^2])^\tau]\,,
\end{align*}
where
\begin{align}
\E[(\bar{T}_{ii}^2-\E[\bar{T}_{ii}^2])^\tau]\nonumber&=\E\Big[\Big(\sum_{t,u=1}^n(\bar{\xi}_{iit}\bar{\xi}_{iiu}-\E[\bar{\xi}_{iit}\bar{\xi}_{iiu}])\Big)^\tau\Big]\nonumber\\
&=\sum_{t_1,u_1,\ldots, t_\tau,u_\tau=1}^n\E[(\bar{\xi}_{iit_1}\bar{\xi}_{iiu_1}-\E[\bar{\xi}_{iit_1}\bar{\xi}_{iiu_1}])\ldots(\bar{\xi}_{iit_\tau}\bar{\xi}_{iiu_\tau}-\E[\bar{\xi}_{iit_\tau}\bar{\xi}_{iiu_\tau}])].\label{gr}
\end{align}
Again the summands with more than $\tau+1$ indices disappear and a summand with $\ell\leq \tau+1$ indices is bounded by

\begin{align*}
C\E[\bar{X}_{it_1}^2\bar{X}_{iu_1}^2\ldots\bar{X}_{it_\tau}^2\bar{X}_{iu_\tau}^2]\lesssim (\E[\bar{X}^2]^{\ell-1}\E[\bar{X}^{4\tau-4\ell+4}])=o((np)^{(4\tau-4\ell+4-\tilde{s})/\tilde{s}})
\end{align*}
by similar arguments as for \eqref{summand}, so that we obtain

\begin{align}
\frac{1}{n^{2\tau}\sigma_n^\tau\delta^\tau}\E\Big[\Big(\sum_{(i,i)\in\Lambda_{n,d}}(\bar{T}_{ii}^2-\E[\bar{T}_{ii}^2])\Big)^\tau\Big]=\frac{1}{n^{2\tau}\sigma_n^\tau\delta^\tau}\max_{1\leq \ell\leq \tau+1}o(n^\ell (np)^{(4\tau-4\ell+4-\tilde{s})/\tilde{s}}).\label{ok}
\end{align}

The term in the maximum is either increasing or decreasing with $\ell$ or it is equal to $o((np)^{4\tau+4-\tilde{s})/\tilde{s}}$, if $n\asymp (np)^{4/\tilde{s}}$. We deduce

\begin{align}
&\frac{1}{n^{2\tau}\sigma_n^\tau\delta^\tau}\max\{o(n^{4\tau/\tilde{s}}p^{4\tau/\tilde{s}-1}),o(n^\tau p^{-1})\}=o\Big(\frac{n^{(s+2)\tau/\tilde{s}} p^{-1}}{n^{2\tau}\sigma_n^\tau\delta^\tau}\Big)=o((n+p)^{(2-\tilde{s})/\tilde{s}}),\label{summary3}
\end{align}
where we used that $(s+2)/\tilde{s}>1$ for an $\varepsilon<1$ and $\sigma_n^2\asymp p/n+(p/n)^3$, if $\E[X^4]\neq 1$.  In the Bernoulli case the second term of \eqref{splt} is zero.\\

In summary we derive by \eqref{summary1}, \eqref{summary2}, \eqref{summary3} and noting that all estimates are uniform in $I_1,\ldots, I_d$ that
\begin{align*}
\max_{I_1,\ldots, I_d} \P(|\bar{Z}_{n,d}|>\delta)=o(n^{-(\tilde{s}-4)\tau/(2\tilde{s})+1})=o(n^{-(s-2)d/2-1}), 
\end{align*}
where $\tau$ can be chosen as the smallest even integer larger than $$\frac{\tilde{s}(s-2)}{\tilde{s}-4}d+\frac{4\tilde{s}}{\tilde{s}-4}.$$
Since the number of possible choices for $I_1<\ldots <I_d$ is $\binom{p(p-1)/2}{d}\leq p^{2d}$ and $p=O(n^{(s-2)/4})$ we conclude that
\begin{align}\label{wsk0}
\sum_{I_1<\ldots<I_d}\P(|\bar{Z}_{n,d}|>\delta)\lesssim \frac{p^{2d}}{n^{(s-2)d/2+1}}\to 0,
\end{align}
as $n\to\infty$, which finishes the proof of \eqref{wk0}.

\section{Auxiliary results}
\label{sec:5}

Throughout this section $p=p_n$ is a sequence of positive integers tending to infinity as $\nto$. Furthermore, let $X,(X_{it})_{i,t\ge 1}$ be iid random variables with $\E[X]=0$ and $\E[X^2]=1$.
\begin{lemma}\label{lem:cut}
 Assume $\E[|X|^s]<\infty$ for some $s\ge 4$. For a positive sequence $(\beta_n)_n$, which tends to zero and satisfies $\beta_n \gg (\E[|X|^s\mathds{1}_{\{|X|>\beta_n(np)^{1/s}\}}])^{1/s}$, set
\begin{align*}
\bar{X}_{it}&:= X_{it}\mathds{1}_{\{|X_{it}|\leq \beta_n (np)^{1/s}\}},\quad 1\leq i\leq p,\,1\leq t\leq n, \\
T_{ij}&:=\sum_{t=1}^n X_{it}X_{jt},\quad \bar{T}_{ij}:=\sum_{t=1}^n\bar{X}_{it}\bar{X}_{jt},\quad 1\leq i,\,j\leq p.
\end{align*}
Then it holds that
$$\frac{1}{n^2\sigma_n}\sum_{1\leq i,j\leq p}(T_{ij}^2-\bar{T}_{ij}^2-\E[T_{ij}^2-\bar{T}_{ij}^2]) \cip 0\,,\qquad \nto\,,$$
where $\sigma_n$ is defined in \eqref{eq:defsigma}.
\end{lemma}
\begin{proof}
We set 
\begin{align*}
Q_n:=\bigcup_{i=1}^p\bigcup_{t=1}^n\{|X_{it}|>\beta_n(np)^{1/s}\}.
\end{align*}
The probability of $Q_n$ tends to zero as $n\to\infty$, since by the union bound and Markov's inequality
\begin{align}
\P(Q_n)\leq np\,\P(|X|>\beta_n (np)^{1/s})&\leq np\,\P\Big(|X|\mathds{1}_{\{|X|>\beta_n (np)^{1/s}\}}>\beta_n (np)^{1/s}\Big)\nonumber\\
&\leq np\,\frac{\E[|X|^s\mathds{1}_{\{|X|>\beta_n (np)^{1/s}\}}]}{\beta_n^s np}\to 0\label{cu2}\,,
\end{align}
where the properties of the sequence $\beta_n$ were used in the last step.
Therefore, we have for $\delta>0$,
\begin{align}
&\P\Big(\frac{1}{n^2\sigma_n}\Big|\sum_{1\leq i,j\leq p}(T_{ij}^2-\bar{T}_{ij}^2-\E[T_{ij}^2-\bar{T}_{ij}^2])\Big|>\delta\Big)\nonumber\\
&\leq \P\Big(\frac{1}{n^2\sigma_n}\Big|\sum_{1\leq i,j\leq p}(T_{ij}^2-\bar{T}_{ij}^2-\E[T_{ij}^2-\bar{T}_{ij}^2])\Big|>\delta, Q_n^C\Big)+\P(Q_n)\nonumber\\
&\leq \P\Big(\frac{1}{n^2\sigma_n}\Big|\sum_{1\leq i,j\leq p}\E[T_{ij}^2-\bar{T}_{ij}^2]\Big|>\delta\Big)+o(1).\label{cu1}
\end{align}
We observe that
\begin{align}\label{momm1}
|\E[\bar{X}]|\leq\frac{\E[|X|^s\mathds{1}_{\{|X|>\beta_n(np)^{1/s}\}}]}{\beta_n^{s-1} (np)^{(s-1)/s}}=o((np)^{-(s-1)/s})
\end{align}
and similarly it follows
\begin{align}\label{momm2}
\E[\bar{X}^2]&=1+o((np)^{-(s-2)/s}) \quad \text{ and } \quad \E[\bar{X}^4]=\E[X^4]+o((np)^{-(s-4)/s}).
\end{align}
Thus we obtain for $1\leq i<j\leq p$,
\begin{align*}
&\E[|T_{ij}^2-\bar{T}_{ij}^2|]=\E[|T_{12}^2-\bar{T}_{12}^2|]\\
&\le n\E[X_{11}^2X_{21}^2-\bar{X}_{11}^2\bar{X}_{21}^2]+n(n-1)\E[|X_{11}X_{21}X_{12}X_{22}-\bar{X}_{11}\bar{X}_{21}\bar{X}_{12}\bar{X}_{22}|]\\
&\lesssim n (1-\E[\bar{X}^2]^2) +n(n-1)\E[|\bar{X}|]^4= o(n^{2/s}p^{-(s-2)/s})
\end{align*}
and as $T_{11}^2-\bar{T}_{11}^2$ is nonnegative we get
\begin{align*}
    \E[|T_{11}^2-\bar{T}_{11}^2|]&=n\E[X^4-\bar{X}^4]+n(n-1)(1-\E[\bar{X}^2]^2)=o(n)+o(n^{1+2/s}p^{-1+2/s}).
\end{align*}
Thereby, it holds 
\begin{align*}
\frac{1}{n^2\sigma_n}\sum_{1\leq i\le j\leq p}\E[|T_{ij}^2-\bar{T}_{ij}^2|]&=o\Big(\frac{n^{2/s-2}p^{1+2/s}}{\sigma_n}\Big)+o\Big(\frac{p}{n\sigma_n}\Big)+o\Big(\frac{n^{2/s-1}p^{2/s}}{\sigma_n}\Big)\,.
\end{align*}
In the case $\E[X^4]\neq 1$, the \rhs\ tends to zero as $\sigma_n\asymp (p/n)^{1/2}+(p/n)^{3/2}$. 
In the symmetric Bernoulli case $X^2=1$, the probability in \eqref{cu1} is zero, establishing the desired result.
\end{proof}

\begin{lemma}\label{lem:var}
Let $\bar{T}_{11}$, $\bar{T}_{12}$ and $\bar{T}_{13}$ be defined as in \eqref{bart}. Under the assumptions of Theorem \ref{thm:clt} it holds, as $\nto$,
\begin{align*}
\Var(\bar{T}_{11}^2)&=4(\E[X^4]-1)n^3+o(n^3)+o(n^2p)\\
\Var(\bar{T}_{12}^2)&=2n^2+((\E[X^4])^2-3)n+o(n)+o(n^{3/2}p^{-1/2})\\
\Cov(\bar{T}_{12}^2,\bar{T}_{13}^2)&=(\E[X^4]-1)n+o(n)\\
\Cov(\bar{T}_{11}^2,\bar{T}_{12}^2)&=2(\E[X^4]-1)n^2+o(n^2)+o(n^{5/2}p^{-1/2}).
\end{align*}

\end{lemma}
 \begin{proof}
Recalling that $\bar{X}=X\1_{\{|X|\leq (np)^{1/4}\beta_n\}}$, we get by \eqref{momm1} and \eqref{momm2} with $s=4$
\begin{align}
|\E[\bar{X}]|= o((np)^{-3/4})\,, \quad \E[\bar{X}^2]=1+o((np)^{-1/2})\,, \quad \E[\bar{X}^4]=\E[X^4]+o(1)\,. \label{mom1}
\end{align}
For higher moments we obtain
\begin{align}
\E[|\bar{X}|^{4+k}]\leq \E[\bar{X}^4](np)^{k/4}\beta_n^k=o((np)^{k/4})\,,\qquad k\geq 1\,.\label{mom3}
\end{align}
After these preparations we will now prove the first claim of the lemma. 
By the multinomial theorem we have
\begin{align*}
\E[\bar{T}_{11}^4]&=\sum_{t_1,t_2,t_3,t_4=1}^n\E[\bar{X}_{1t_1}^2\bar{X}_{1t_2}^2\bar{X}_{1t_3}^2\bar{X}_{1t_4}^2]\\
&=n(n-1)(n-2)(n-3)\E[\bar{X}^2]^4+6n(n-1)(n-2)\E[\bar{X}^4]\E[\bar{X}^2]^2\\
&\quad +3n(n-1)\E[\bar{X}^4]^2+4n(n-1)\E[\bar{X}^6]\E[\bar{X}^2]+n\E[\bar{X}^8]\\
&=n^4+6n^3(\E[X^4]-1)+o(n^3)+o(n^2p),
\end{align*}
where we used \eqref{mom1} and \eqref{mom3} in the last step. The same arguments also yield
\begin{align}\label{t11}
\E[\bar{T}_{11}^2]= n\E[\bar{X}^4]+n(n-1)\E[\bar{X}^2]^2
&=n^2+n(\E[X^4]-1)+o(n).
\end{align}
Since $\E[\bar{T}_{11}^2]^2=n^4+2n^3(\E[X^4]-1)+o(n^3)$, we conclude that  
\begin{align*}
\Var(\bar{T}_{11}^2)=\E[\bar{T}_{11}^4]-E[\bar{T}_{11}^2]^2=4(\E[X^4]-1)n^3+o(n^3)+o(n^2p),
\end{align*}
which proves the first part of the lemma. For the second part we consider $\Var(\bar{T}_{12}^2)=\E[\bar{T}_{12}^4]-\E[\bar{T}_{12}^2]^2$. Using \eqref{mom1} and \eqref{mom3}, we get
\begin{align*}
\E[\bar{T}_{12}^4]
&=n(n-1)(n-2)(n-3)\E[\bar{X}]^8+6n(n-1)(n-2)\E[\bar{X}^2]^2\E[\bar{X}]^4\\
&\quad +3n(n-1)\E[\bar{X}^2]^4+4n(n-1)\E[\bar{X}^3]^2\E[\bar{X}]^2+n\E[\bar{X}^4]^2\\
&=3n^2-3n+n\E[X^4]^2+o(n)+o(n^{3/2}p^{-1/2}).
\end{align*}
as well as 
\begin{align}\label{t12}
\E[\bar{T}_{12}^2]=\sum_{t_1,t_2=1}^n\E[\bar{X}_{1t_1}\bar{X}_{2t_1}\bar{X}_{1t_2}\bar{X}_{2t_2}]=n\E[\bar{X}^2]^2+n(n-1)\E[\bar{X}]^4=n+o(1)\,,
\end{align}
which implies $\E[\bar{T}_{12}^2]^2=n^2+o(n)$. Since $\Var(\bar{T}_{12}^2)=\E[\bar{T}_{12}^4]-\E[\bar{T}_{12}^2]^2$ the second part of the lemma is established.

To show part three of the lemma, we compute, using \eqref{mom1}, that
\begin{align*}
\E[\bar{T}_{12}^2\bar{T}_{13}^2]&=\sum_{t_1,t_2,t_3,t_4=1}^n\E[\bar{X}_{1t_1}\bar{X}_{1t_2}\bar{X}_{1t_3}\bar{X}_{1t_4}]\E[\bar{X}_{2t_1}\bar{X}_{2t_2}]\E[\bar{X}_{3t_3}\bar{X}_{3t_4}]\\
&=n(n-1)(n-2)(n-3)\E[\bar{X}]^8+2n(n-1)(n-2)\E[\bar{X}]^4\E[\bar{X}^2]^2\\
&\quad +4n(n-1)(n-2)\E[\bar{X}^2]\E[\bar{X}]^6+n(n-1)\E[\bar{X}^2]^4\\
&\quad +2n(n-1)\E[\bar{X}^2]^2\E[\bar{X}]^4+4n(n-1)\E[\bar{X}^3]\E[\bar{X}^2]\E[\bar{X}]^3+n\E[\bar{X}^4]\E[\bar{X}^2]^2\\
&=n^2+n(\E[X^4]-1)+o(n).
\end{align*}
In conjunction  with $\eqref{t12}$, we then obtain
\begin{align*}
\Cov(\bar{T}_{12}^2,\bar{T}_{13}^2)=\E[\bar{T}_{12}^2\bar{T}_{13}^2]-\E[\bar{T}_{12}^2]\E[\bar{T}_{13}^2]=n(\E[X^4]-1)+o(n).
\end{align*}

Finally, using \eqref{mom1} and \eqref{mom3}, we have 
\begin{align*}
\E[\bar{T}_{11}^2\bar{T}_{12}^2]&=\sum_{t_1=1}^n\sum_{t_2=1}^n\sum_{t_3=1}^n\sum_{t_4=1}^n\E[\bar{X}_{1t_1}^2\bar{X}_{1t_2}^2\bar{X}_{1t_3}\bar{X}_{1t_4}]\E[\bar{X}_{2t_3}\bar{X}_{2t_4}]\\
&=n(n-1)(n-2)(n-3)\E[\bar{X}^2]^2\E[\bar{X}]^4+n(n-1)(n-2)\E[\bar{X}^4]\E[\bar{X}]^4\\
&\quad +n(n-1)(n-2)\E[\bar{X}^2]^4+4n(n-1)(n-2)\E[\bar{X}^3]\E[\bar{X}^2]\E[\bar{X}]^3\\
&\quad +n(n-1)\E[\bar{X}^4]\E[\bar{X}^2]^2+2n(n-1)\E[\bar{X}^3]^2\E[\bar{X}]^2 + 2n(n-1)\E[\bar{X}^5]\E[\bar{X}]^3\\
&\quad +2n(n-1)\E[\bar{X}^4]\E[\bar{X}^2]^2+n\E[\bar{X}^6]\E[\bar{X}^2]\\
&=n^3+3(\E[X^4]-1)n^2+o(n^2)+o(n^{5/2}p^{-1/2}).
\end{align*}
Therefore, we obtain in combination with \eqref{t11} and \eqref{t12},
\begin{align*}
\Cov(\bar{T}_{11}^2,\bar{T}_{12}^2)=\E[\bar{T}_{11}^2\bar{T}_{12}^2]-\E[\bar{T}_{11}^2]\E[\bar{T}_{12}^2]=2(\E[X^4]-1)n^2+o(n^2)+o(n^{5/2}p^{-1/2})\,,
\end{align*}
completing the proof of the lemma.
\end{proof}

\begin{lemma}\label{lem:zw}
Let $(\bar{T}_{ij})$ be defined as in \eqref{bart} an write  $\mathcal{F}_{j}$ for the sigma algebra generated by $\{\tx_1,\ldots,\tx_j\}$, where $\tx_i=(X_{i1}, \ldots, X_{in})$. Under the conditions of Theorem \ref{thm:clt} it holds for $1\le i_1,i_2<j\le p$ that
\begin{align*}
    &\E[\bar{T}_{i_1j}^2\bar{T}_{i_2j}^2|\mathcal{F}_{j-1}]-\E[\bar{T}_{i_1j}^2|\mathcal{F}_{j-1}]\E[\bar{T}_{i_1j}^2|\mathcal{F}_{j-1}]\\
    &=(\E[\bar{X}^4]-(\E[\bar{X}^2])^2)\sum_{t=1}^n\bar{X}_{i_1t}^2\bar{X}_{i_2t}^2\\
    &\quad +\E[\bar{X}](2\E[\bar{X}^3]-\E[\bar{X}^2])\sum_{\substack{t_1,t_2=1\\t_1\neq t_2}}^n \big(\bar{X}_{i_1t_1}^2\bar{X}_{i_2t_1}\bar{X}_{i_2t_2}+\bar{X}_{i_2t_1}^2\bar{X}_{i_1t_1}\bar{X}_{i_1t_2}\big)\\
    &\quad +2((\E[\bar{X}^2])^2-\E[\bar{X}]^4)\sum_{\substack{t_1,t_2=1\\t_1\neq t_2}}^n\bar{X}_{i_1t_1}\bar{X}_{i_1t_2}\bar{X}_{i_2t_1}\bar{X}_{i_2t_2}\\
    &\quad +4\E[\bar{X}]^2(\E[\bar{X}^2]-\E[\bar{X}]^2)\sum_{\substack{t_1,t_2,t_3=1\\|\{t_1,t_2,t_3\}|=3}}^n\bar{X}_{i_1t_1}\bar{X}_{i_1t_2}\bar{X}_{i_2t_1}\bar{X}_{i_2t_2}.
\end{align*}
\end{lemma}
\begin{proof}
By straightforward calculation we get
\begin{align}
    &\E[\bar{T}_{i_1j}^2\bar{T}_{i_2j}^2|\mathcal{F}_{j-1}]=\E[\bar{X}]^4\sum_{\substack{t_1,\ldots,t_4=1\\|\{t_1,\ldots, t_4\}|=4}}^n\bar{X}_{i_1t_1}\bar{X}_{i_1t_2}\bar{X}_{i_2t_3}\bar{X}_{i_2t_4}\nonumber\\
    &+\E[\bar{X}^2]\E[\bar{X}]^2\sum_{\substack{t_1,t_2,t_3=1\\t_1\neq t_2\neq t_3}}^n\big(\bar{X}_{i_1t_1}^2\bar{X}_{i_2t_2}\bar{X}_{i_2t_3}+\bar{X}_{i_2t_1}^2\bar{X}_{i_1t_2}\bar{X}_{i_1t_3}+4\bar{X}_{i_1t_1}\bar{X}_{i_1t_2}\bar{X}_{i_2t_1}\bar{X}_{i_2t_3}\big)\nonumber\\
    &+\E[\bar{X}^2]^2\sum_{\substack{t_1,t_2=1\\t_1\neq t_2}}^n\big(\bar{X}_{i_1t_1}^2\bar{X}_{i_2t_2}^2+2\bar{X}_{i_1t_1}\bar{X}_{i_1t_2}\bar{X}_{i_2t_1}\bar{X}_{i_2t_2}\big)\nonumber\\
    &+2\E[\bar{X}]\E[\bar{X}^3]\sum_{\substack{t_1,t_2=1\\t_1\neq t_2}}^n\big(\bar{X}_{i_1t_1}^2\bar{X}_{i_2t_1}\bar{X}_{i_2t_2}+\bar{X}_{i_2t_1}^2\bar{X}_{i_1t_1}\bar{X}_{i_1t_2}\big)\nonumber\\
    &+(\E[\bar{X}^4]-\E[\bar{X}^2]^2)\sum_{t=1}^n\bar{X}_{i_1t}^2\bar{X}_{i_2t}^2.\label{pa1}
\end{align}

Additionally it holds that
\begin{align*}
 \E[\bar{T}_{i_1j}^2|\mathcal{F}_{j-1}]=\E[\bar{X}^2]\sum_{t=1}^n\bar{X}_{i_1t}^2 +\E[\bar{X}]^2\sum_{\substack{t_1,t_2=1\\t_1\neq t_2}}^n \bar{X}_{i_1t_1}\bar{X}_{i_1t_2} 
\end{align*}
and therefore 
\begin{align}
    &\E[\bar{T}_{i_1j}^2|\mathcal{F}_{j-1}]\E[\bar{T}_{i_2j}^2|\mathcal{F}_{j-1}]\nonumber\\
    &=\E[\bar{X}^2]^2\sum_{t_1,t_2=1}^n\bar{X}_{i_1t_1}^2\bar{X}_{i_2t_2}^2+\E[\bar{X}^2]\E[\bar{X}]^2\sum_{\substack{t_1,t_2,t_3=1\nonumber\\ t_2\neq t_3}}\bar{X}_{i_1t_1}^2\bar{X}_{i_2t_2}\bar{X}_{i_2t_3}+\bar{X}_{i_2t_1}^2\bar{X}_{i_1t_2}\bar{X}_{i_1t_3}\nonumber\\
    &+\E[\bar{X}]^4\sum_{\substack{t_1,\ldots t_4=1\\ t_1\neq t_2,\, t_3\neq t_4}}^n\bar{X}_{i_1t_1}\bar{X}_{i_1t_2}\bar{X}_{i_2t_3}\bar{X}_{i_2t_4}.\label{pa2}
\end{align}
The lemma follows from \eqref{pa1} and \eqref{pa2}.
\end{proof}

\begin{lemma}\label{lem:hilf}
Let $Y$ and $Y'$ be real-valued random variables and $B$ an arbitrary event. Then it holds for every $y\in\R$ and $\delta>0$ that
\begin{align*}
\P(Y\leq y,\, B)&\leq \P(Y-Y'\leq y+\delta,\, B)+\P(|Y'|>\delta)\,,\\
\P(Y\leq y,\, B)&\geq \P(Y-Y'\leq y-\delta,\, B)-\P(|Y'|>\delta)\,.
\end{align*}
\end{lemma}

\begin{proof}
We have
\begin{align*}
\P(Y\leq y,\, B)&\leq \P(Y\leq y,\, B,\, |Y'|\leq \delta)+\P(|Y'|>\delta)\\
&\leq \P(Y-Y'\leq y+\delta,\, B)+\P(|Y'|>\delta)
\end{align*}
and 
\begin{align*}
\P(Y\leq y,\, B)+\P(|Y'|>\delta)&\geq \P(Y-Y'\leq y-\delta,\, B,\, |Y'|\leq \delta)+ \P(|Y'|>\delta)\\
&\geq \P(Y-Y'\leq y-\delta,\, B).
\end{align*}
\end{proof}
\bibliography{libraryjohannes}
\end{document}